\documentclass{amsart}
\usepackage{amsmath,mathtools}
\usepackage{amsrefs}
\usepackage{microtype}
\usepackage{lmodern}
\usepackage{hyperref}
\usepackage{tikz-cd}
\usepackage{graphicx}
\usepackage{tikz-network}% Required for inserting images
\usetikzlibrary{arrows}
\usepackage{enumerate}
\usepackage{parskip}

\usetikzlibrary{decorations.markings}
\usepackage{caption}

\theoremstyle{plain}
\newtheorem{theorem}[subsection]{Theorem}
\newtheorem{proposition}[subsection]{Proposition}
\newtheorem{lemma}[subsection]{Lemma}

\theoremstyle{definition}
\newtheorem{definition}[subsection]{Definition}
\newtheorem{example}[subsection]{Example}

\theoremstyle{remark}
\newtheorem{remark}[subsection]{Remark}

\numberwithin{equation}{section}

\newcommand{\Dlt}{\Delta}

\newcommand{\cal}{\mathcal}
\newcommand{\cla}{{\cal A}}
\newcommand{\clb}{{\cal B}}
\newcommand{\clc}{{\cal C}}
\newcommand{\cld}{{\cal D}}
\newcommand{\cle}{\cal{E}}

\newcommand{\clh}{{\cal H}}

\newcommand{\clk}{{\cal K}}
\newcommand{\cll}{{\cal L}}
\newcommand{\clm}{{\cal M}}

\newcommand{\clo}{{\cal O}}

\newcommand{\clr}{{\cal R}}

\def\a*{{\cal A}_{h,*}}
\def\B{{\cal B}(h)}
\def\B1{{\cal B}_1(h)}
\def\b{{\cal B}^{\rm s.a.}(h)}
\def\b1{{\cal B}^{\rm s.a.}_1(h)}

\newcommand{\raro}{\rightarrow}

\def \qed {$\Box$}

\def\a*{{\cal A}_{h,*}}
\def\B1{{\cal B}_1(h)}
\def\b{{\cal B}^{\rm s.a.}(h)}
\def\b1{{\cal B}^{\rm s.a.}_1(h)}

\usepackage{graphicx} % Required for inserting images
\begin{document} 
\title[]{Isometric actions of compact quantum groups on graph $C^{\ast}$-algebras}
\author[Joardar]{Soumalya Joardar}
\address{Department of Mathematics and Statistics, Indian Institute of Science Education and Research Kolkata, Mohanpur - 741246, West
	Bengal, India}
\email{soumalya@iiserkol.ac.in}
\author[Sharma]{Jitender Sharma}
\address{Department of Mathematics and Statistics, Indian Institute of Science Education and Research Kolkata, Mohanpur - 741246, West
	Bengal, India}
\email{js20ip007@iiserkol.ac.in}

\begin{abstract}
For a strongly connected directed graph $\Lambda$, it is shown that the quantum automorphism group of $\Lambda$ acts isometrically in the sense of D. Goswami et al. (\cite{goswami1}) on the spectral triple $(C^{\ast}(\Lambda), L^{2}(\Lambda^{\infty},M),D)$ constructed by Farsi et al.. It is also shown that the natural action of $U_{n}^{+}$ is not isometric on the Cuntz algebra $\clo_{n}$.
\end{abstract}
\maketitle
\section{Introduction}
Given a mathematical structure, to understand its symmetry is a fundamental problem. It often helps in reducing the complexity of a system exhibiting a given mathematical structure. Classically such symmetries are modelled by groups. For example, given a Riemannian manifold $M$, a symmetry of the manifold should be given by some group which acts on $M$ preserving the Riemannian metric. In the realm of noncommutative geometry or topology, a `compact' space is typically replaced by a unital $C^{\ast}$-algebra or von Neumann algebra. Then it becomes very important to understand `generalized symmetry' of such algebras. By now it has become standard to replace groups by what are known as compact quantum groups (CQG's in short). Since Wang in his seminal paper (see \cite{wang}) showed the existence of a genuine quantum symmetry for a space consisting of finite points (at least 4 points), CQG's have fulfilled its role as quantum symmetry object in the realm of noncommutative geometry very successfully. To study more geometric properties, A. Connes introduced the concept of a spectral triple on $C^{\ast}$-algebras. A spectral triple is a triple $(\cla,\clh,D)$ where typically $\cla$ is a $C^{\ast}$-algebra represented faithfully on a Hilbert space $\clh$ and $D$ is some unbounded operator, called the Dirac operator, capturing the `geometric' data. Quite naturally one seeks for a notion of an `isometric' action of CQG's on such spectral triples which should be a generalization of classical isometric group actions on Riemannian manifolds. This is done in a series of papers by D. Goswami and his collaborators (see \cite{goswami1}, \cite{goswami2}, \cite{goswami3} for example). In these papers, the authors develop a notion of a `quantum isometric' action of CQG's on spectral triples both in terms of the Dirac operator of the spectral triple and in terms of a suitable Laplacian operator derived from the spectral triple data. \\
\indent A very important and tractable class of $C^{\ast}$-algebras are given by what are known as graph $C^{\ast}$-algebras. As the name hints, these $C^{\ast}$-algebras are built from directed graphs. Due to the underlying combinatorial structure of the graph, this class has become very well studied class of $C^{\ast}$-algebras (see \cite{raeburn},\cites{kumjian_pask_raeburn}). This underlying combinatorial data has also been useful in studying quantum symmetry of the graph $C^{\ast}$-algebras. For example, the quantum symmetry of the underlying graph naturally acts on the corresponding graph $C^{\ast}$-algebra (see \cite{weber}, \cite{joardar1}, \cite{joardar2}, \cite{joardar3}, \cite{arnab},\cite{joardar_canadian} for quantum group actions on graph $C^{\ast}$-algebras). Recently, a spectral triple has been introduced on the graph $C^{\ast}$-algebras (see \cite{marcolli}, \cite{farsi}). In these spectral triples, the graph $C^{\ast}$-algebras corresponding to strongly connected finite directed graphs are faithfully represented on an $L^2$-space of infinite path spaces completed with respect to a measure derived from its unique KMS-state. So it becomes quite natural to understand `quantum isometric' actions on such spectral triples. In this paper, we show that given a strongly connected finite graph $\Lambda$ without multiple edge, loop or source and a spectral triple on the corresponding graph $C^{\ast}$-algebra, the natural action of the quantum automorphism group of the underlying graph (to be denoted by ${\rm Aut}^{+}(\Lambda)$) is `isometric' in the sense of Goswami et al.(\cite{goswami1}). The essential step is to construct a unitary co-representation $\widetilde{U}$ of the CQG on the $L^2$ space of infinite paths that can implement the given action. To prove the unitarity of $\widetilde{U}$, the invariance of the unique KMS-state under the action of the ${\rm Aut}^{+}(\Lambda)$ becomes important. That leads us naturally to consider the universal object in the category of CQG's acting linearly and preserving the KMS state on a graph $C^{\ast}$-algebra. In case of the Cuntz algebra with $n$-generators, it is well known that the quantum unitary group $U_{n}^{+}$ is the universal object in that category. We show that the action of $U_{n}^{+}$ is not isometric however. This makes ${\rm Aut}^{+}(\Lambda)$ a natural general class of quantum groups that can act `quantum isometrically' on the corresponding graph $C^{\ast}$-algebras. Note that although isometric group action has been studied on Cuntz algebras before (see \cites{rossi, park}), the study of quantum isometries for a general class of graph $C^{\ast}$-algebras is certainly new and we hope that this work will give impetus to understand quantum isometries of graph $C^{\ast}$-algebras.  \vspace{0.1in}\\
\textbf{Notations}: In this paper, we will use some standard notations. ${\rm Sp}\{A\}$ will stand for the linear span of elements of $A$ in a given vector space; $\mathcal{H}\overline{\otimes}\mathcal{A}$ will refer to the external tensor product of Hilbert space with a $C^{\ast}$-algebra to form a Hilbert-$\cla$ module. We shall use $\mathcal{B}( \mathcal{H}), \mathcal{K}(\mathcal{H})$ to denote the set of bounded and compact linear operators on a Hilbert space $\mathcal{H}$ respectively. If $\cle$ is a Hilbert $\mathcal{A}-$module then $\cll(\cle)$ is the $C^{\ast}$-algebra of adjointable $\mathcal{A}$-linear maps. $\otimes$ will denote the minimal tensor product between $C^{\ast}$-algebras as well as tensor product of Hilbert spaces whereas $\otimes_{\rm{alg}}$ will stand for algebraic tensor product. We use $\langle\;,\rangle$ for both the $C^{\ast}$-algebra valued inner product on a Hilbert module and a $\mathbb{C}$-valued inner product on a Hilbert space. For a $C^{\ast}$-algebra $\cla$, $\clm(\cla)$ will denote its multiplier algebra.
\section{Preliminaries}
\subsection{Compact quantum groups}
In this section we shall recall the preliminaries on compact quantum groups and their actions on $C^{\ast}$-algebras. The reader is referred to any of \cite{woronowicz1}, \cite{vandaele} for a thorough exposition on compact quantum groups. For generalities on actions of CQG on $C^{\ast}$-algebras the reader can consult \cite{decommer}.  
% Definition of compact quantum group
\begin{definition}
    A compact quantum group $\mathbb{G}$ is a pair $(C(\mathbb{G}),\Dlt)$ where $C(\mathbb{G})$ is a unital $C^{\ast}$-algebra and $\Dlt: C(\mathbb{G})\raro C(\mathbb{G})\otimes C(\mathbb{G})$ is a unital $C^{\ast}$-homomorphism which satisfies the following conditions: 
\begin{enumerate}[(I)]
\setlength\itemsep{0.5em}
    \item $(\rm{id}\otimes \Dlt)\circ \Dlt=(\Dlt \otimes \rm{id})\circ \Dlt$
    \item $\text{Sp}\{\Dlt(C(\mathbb{G}))(1\otimes C(\mathbb{G}))\}$ and $\text{Sp}\{\Dlt(C(\mathbb{G}))(C(\mathbb{G}) \otimes 1)\}$ are dense in $C(\mathbb{G})\otimes C(\mathbb{G})$. 
\end{enumerate}
\end{definition}
Given a CQG $\mathbb{G}$, there is a canonical dense Hopf $\ast$-algebra $C(\mathbb{G})_{0}$ of $C(\mathbb{G})$ where the coproduct is algebraic i.e. $\Delta(C(\mathbb{G})_{0})\subset C(\mathbb{G})_{0}\otimes_{\rm alg}C(\mathbb{G})_{0}$. Given any CQG $\mathbb{G}$, $C(\mathbb{G})_{0}$ will denote this canonical dense Hopf $\ast$-algebra.
% Quantum permutation group
\begin{example}[{Quantum permutation group}] Consider the universal $C^{\ast}$-algebra generated by a magic unitary of order $n$, denoted by $C(S_{n}^{+})$ i.e. 
     \begin{displaymath}C(S_{n}^{+}):=C^{\ast}\left(q_{ij}, 1\leq i,j\leq n|\; q_{ij}^2=q_{ij}=q_{ij}^{\ast}\; \text{and} \; \sum\limits_{k}q_{ik}=\sum\limits_{k}q_{kj}=1\right).\end{displaymath} With the coproduct map on the generators of $C(S_{n}^{+})$ given by $\Dlt(q_{ij})=\sum\limits_{k}q_{ik}\otimes q_{kj}$,  $(C(S_{n}^{+}),\Delta)$ becomes a compact quantum group. It is denoted by $S_{n}^{+}$ and called the quantum permutation group on $n$-points.
\end{example}
% Qunautm unitary group
\begin{example}[Quantum unitary group] Consider the universal $C^{\ast}$-algebra generated by unitary of order $n$, denoted by $C(U_{n}^{+})$, that is, 
\begin{displaymath}C(U_{n}^{+}):=C^{\ast}\Biggl(q_{ij}, 1\leq i,j\leq n|\; q=(q_{ij})\;\text{and} \;\overline{q}=(q_{ij}^{\ast})\; \text{are unitaries}\Biggl).\end{displaymath}
The co-product map on the generators given by $\Dlt(q_{ij})=\sum\limits_{k}q_{ik}\otimes q_{kj}$. Then $(C(U_{n}^{+}),\Dlt)$ is a compact quantum group. This is called the quantum unitary group and is denoted by $U_{n}^{+}$.
\end{example}
\begin{example}[Qaut of a finite directed graph] (see \cite{banica}) Let $\Lambda=(V,E,r,s)$ be a finite directed graph without loops or multiple edges, where $V$ and $E$ are the sets of vertices and edges respectively and $r,s$ are the standard range and source maps. The quantum automorphism group of $\Lambda$ is the pair $(C({\rm Aut}^{+}(\Lambda)),\Delta)$ where $C({\rm Aut}^{+}(\Lambda))$ is the universal $C^{\ast}$-algebra generated by the elements $\{q_{ij}\}_{i,j\in V}$ such that the matrix $U=((q_{ij}))_{i,j\in V}$ is a magic unitary of order $|V|$ and $UA=AU$ where $A$ is the adjacency matrix of $\Lambda$. The co-product $\Delta$ is given on the generators by $\Delta(q_{ij})=\sum_{k}q_{ik}\otimes q_{kj}$. One can write down all the relations between the generators explicitly. In particular, we note down the following relations for future purpose (see \cite{fulton})
    \begin{eqnarray}
    \label{qaut_equ}
        q_{s(\gamma)i}q_{r(\gamma)k} =q_{r(\gamma)k} q_{s(\gamma)i}&=0,\; \gamma\in E,\; (i,k)\notin E \nonumber \\
        q_{is(\gamma)}q_{kr(\gamma)} =q_{kr(\gamma)} q_{is(\gamma)}&=0,\; \gamma\in E,\; (i,k)\notin E.
    \end{eqnarray}
\end{example}
% Compact quantum group action on unital C^*-algebra 
\begin{definition}
    Given a unital $C^{\ast}$-algebra $\cla$, a compact quantum group $\mathbb{G}$ is said to act on $\cla$ if there is a unital $C^{\ast}$-homomorphism $\alpha:\cla\raro\cla\otimes C(\mathbb{G})$ such that
    \begin{enumerate}[(I)]
    \setlength\itemsep{0.5em}
    \item $(\rm{id}\otimes \Dlt)\circ \alpha=(\alpha \otimes \rm{id})\circ \alpha$.
    \item ${\rm Sp}\{\alpha(\cla)(1\otimes C(\mathbb{G}))\}$ is dense in $\cla\otimes C(\mathbb{G})$.
\end{enumerate}
\end{definition}
Moreover, a state $\phi$ on a unital $C^{\ast}$-algebra $\mathcal{A}$ is said to be \textit{invariant} under $\alpha$ if  
    $(\phi\otimes {\rm id})\circ \alpha(a)=\phi(a).1$ for all $a\in \mathcal{A}$.

Now let us recall the rudiments of graph $C^{\ast}$-algebras and in particular the actions of the quantum automorphism groups of the underlying graphs on them. This is going to be the central object of this paper. The reader is referred to \cite{raeburn} and \cites{weber, joardar1} for more details on this topic. Let $\Lambda=(V,E,r,s)$ be a strongly connected finite directed graph without multiple edges, loops and sources. Recall that by a strongly connected graph, we mean a graph where any two vertices can be joined by a path of finite length. \\

    \indent We remark that if $\Lambda=(V,E,r,s)$ is a finite directed graph without loops and multiple edges, we say two edges $e,f$ are composable if $s(e)=r(f)$ and write it as $ef$. More generally, for two paths $\lambda,\mu$, we say $\lambda,\mu$ are composable if $s(\lambda)=r(\mu)$ and write the composed path as $\lambda\mu$ (see \cite{farsi}). Also note that the use of the term `loop' is different from that in \cite{farsi} in the context of higher rank graph. In \cite{farsi}, a loop is a path $\lambda$ of finite length such that $r(\lambda)=s(\lambda)$. However, we use the term `loop' to represent an edge $e$ such that $r(e)=s(e)$. In this paper we call a path $\lambda$ of length greater than $1$ such that $r(\lambda)=s(\lambda)$ a closed path. Thus for us a closed path of length $1$ is a loop. Given a finite path $\lambda$, we denote the length of the path by $d(\lambda)$. 

%\textbf{Perron Frobenious result}: A non-negative matrix $A\in M_{n}(\mathbb{C})$ is said to irreducible if there exists a finite subset  $F\in \mathbb{N}$ such that $\sum_{n\in F}A^{n}$ is positive. \\
%Perron Frobenious result says that for an irreducible matrix $A\in M_{n}(\mathbb{C})$ spectral radius $\rho(A)$ is an eigenvalue of $A$ with $1-$dimensional eigenspace and unimodular positive eigenvector. We call this positive unimodular eigenvector as \textit{Perron frobenious vector for the given irreducible matrix}. We note that the vertex matrix for strongly connected finite graph the corresponding vertex matrix is irreducible. 
\begin{definition}
Let $\Lambda=(V,E,r,s)$ be a row finite directed graph consisting of a  countable vertex set $V$, countable edge set $E$ and the maps $r,s:E\raro V$ identifying the range and source of each edge, then $C^{\ast}(\Lambda)$ is defined as the universal $C^{\ast}$ algebra generated by mutually orthogonal projections $\{p_{v},v\in V\}$ and partial isometries $\{S_{e}, e\in E\}$ that satisfy,
\begin{enumerate}[(I)]
     \setlength\itemsep{0.5em}
    \item $S_{e}^{\ast}S_{e}=p_{s(e)}$ for all $e\in E$; and 
    \item $p_{v}=\sum\limits_{\{e\in E,\; r(e)=v\}}S_{e}S_{e}^{\ast}$.
\end{enumerate}
\end{definition}
It follows immediately that for a finite graph $\Lambda$, $\sum\limits_{v} p_{v}=1$ and therefore the graph $C^{\ast}$-algebra $C^{\ast}(\Lambda)$ is unital. In this paper $\Lambda$ will always be finite. For a finite directed graph $\Lambda=(V,E,r,s)$ without loops and multiple edges and sources, let $V=\{1,2,\cdots,n\}$ and $E=\{e_{1},e_{2},\cdots,e_{m}\}$. For the proof of the following proposition, the reader is referred to  \cites{weber, joardar1}. 
\begin{proposition}
\label{actionofqaut}
Let $\Lambda=(V,E,r,s)$ be a finite directed graph without loops or multiple edges or sources and $C^{\ast}(\Lambda)$ be the corresponding graph $C^{\ast}$ algebra. Then the quantum automorphism group $Aut^{+}(\Lambda)$ acts on $C^{\ast}(\Lambda)$. The action $\alpha:C^{\ast}(\Lambda)\raro C^{\ast}(\Lambda)\otimes C(Aut^{+}(\Lambda))$ is given by, 
\begin{equation*}
   \alpha(p_{i})=\sum\limits_{k=1}^{n}p_{k}\otimes q_{ki}\;\;\;  1\leq i\leq n
\end{equation*}
and 
\begin{equation*}
   \alpha(S_{e_{j}})=\sum\limits_{l=1}^{m}S_{e_{l}}\otimes q_{r(e_{l})r(e_{j})}q_{s(e_{l})s(e_{j})} \;\;\; 1\leq j\leq m.
\end{equation*}
\end{proposition}
\begin{remark}
    The above action is slightly different form the action defined by Schmidt and Weber.  The difference lies in the coefficients of $S_{e_{l}}$. In \cite{weber}, the coefficients are taken as ``$q_{s(e_{j})s(e_{l})}q_{r(e_{j})r(e_{l})}$''. This difference arises from the fact that we are using slightly different convention for composability of paths.
    %We are following \cite{farsi}'s convention for composability of paths i.e. for us, given two edges $e,f$, $ef$ is a path if $r(f)=s(e)$. 
\end{remark}
There is a very important class of states on graph $C^{\ast}$-algebras known as the KMS states. We are not going to recall the basics of KMS states. The reader is referred to \cite{KMS} for details on KMS states on graph $C^{\ast}$-algebras. It is well known that for a strongly connected graph $\Lambda$, the graph $C^{\ast}$-algebra $C^{\ast}(\Lambda)$ has a unique KMS state at the critical inverse temperature. If we denote the unique KMS state by $\phi$, then $\phi$ is given on the vertex projections by $\phi(p_{v})=x^{\Lambda}_{v}$, where $\{x^{\Lambda}_{v}\}_{v\in V}$ is the normalized Perron-Frobenius vector of the vertex matrix (see \cite{seperable representation}). Moreover $\phi$ is invariant under the action of ${\rm Aut}^{+}(\Lambda)$ (Theorem 3.1 of \cite{joardar2}). 
%We are going to record the following consequence of this invariance which will be useful later on.
%\begin{lemma}
%\label{KMSstateinvariance}
%Let $\Lambda$ be a strongly connected graph without multiple edge, loop or source. Let the vertex set be given by $V=\{1,2,\ldots,n\}$; $\{x^{\Lambda}_{i}\}_{i\in V}$ be the unique normalized Perron-Frobenius eigenvector. Also let $\alpha$ be the action of ${\rm Aut}^{+}(\Lambda)$ on $C^{\ast}(\Lambda)$ as given in Proposition \ref{actionofqaut}. Then $q_{ij}=0$ if $x^{\Lambda}_{i}\neq x^{\Lambda}_{j}$ for $i,j\in V$.   
%\end{lemma}
%\begin{proof}
%Combine Lemma 2.32 and Theorem 3.1 of \cite{joardar2}    
%\end{proof}
%\begin{proposition}
 %   Let $\Lambda$ be strongly connected finite directed graph without multiple edges and loops. Let $\gamma:\mathbb{T}\raro Aut(C^{\ast}(\Lambda))$ be the gauge action and define $\beta:\mathbb{R}\raro Aut(C^{\ast}(\Lambda))$ by $\beta_{t}:=\gamma_{it}$. Let $x$ be the unique unimodular Perron Frobenious vector of the vertex matrix $A$. Then the system $(C^{\ast}(\Lambda),\beta)$ has a unique $KMS_{ln(\rho(A))}$ state $\phi$ which satisfies, 
%\begin{equation*}
 %   \phi(S_{\mu}S_{\nu}^{\ast})=\delta_{\mu,\nu}\rho(A)^{-d(\mu)}x_{s(\mu)}^{\Lambda}
%\end{equation*}
%\end{proposition}
%\begin{proof}
 %   \cite{Huef, Laca, Rarburn, Sims}
%\end{proof}
%\begin{remark}
%\end{remark}
\subsection{Spectral triples and isometric actions} We start this subsection by recalling the definition of a spectral triple.
\begin{definition}
    A spectral triple $(\mathcal{A},\mathcal{H},D)$ consists of ${\ast}$-algebra $\mathcal{A}$ with a faithful representation
\begin{equation*}
    \pi: \mathcal{A}\raro \mathcal{B}(\mathcal{H})
\end{equation*}
where $\clh$ is a Hilbert space and a densely defined operator $D$(possibly unbounded) called the \textit{Dirac operator} which satisfies the following conditions:
\begin{enumerate}[(I)]
\setlength\itemsep{0.5em}
    \item $D$ is self-adjoint;
    \item For all $c\in \mathbb{R}$ in the resolvent, $(D-c)^{-1}\in \mathcal{K}(\mathcal{H})$;
\item For all $a\in \mathcal{A}, [D,a]$ is a bounded operator on $\mathcal{H}$.
\end{enumerate}
\end{definition}

Now let us recall the spectral triple introduced in \cite{farsi} for a strongly connected finite $k$-graph very briskly. We are going to specialize the construction for $k=1$. To that end, let $\Lambda=(V,E,r,s)$ be a finite directed graph without sources. We remark that for the spectral triple discussion we don't need to restrict the class of graphs without multiple edges or loops. Let $\Lambda^{\ast}=\cup_{n\geq 0} E^{n}$ be the collection of all finite paths of $\Lambda$ with range and source maps extended accordingly. This along with the usual length function $d:E^{\ast}\raro \mathbb{N}$., forms a $k$-graph with $k=1$. We shall consider the set of all infinite paths,
\begin{displaymath}\Lambda^{\infty}=\{\lambda=(\lambda_{1}\lambda_{2}\cdots)|\; r(\lambda_{i+1})=s(\lambda_{i}), i=1,2,\cdots; \lambda_{1}\lambda_{2}\cdots \text{is an infinite sequence}\}\end{displaymath}
This infinite path space $\Lambda^{\infty}$ being a subset of $\prod E$ has the natural subspace topology from product space for which the \textit{cylinder sets} ,
$$[\lambda]=\{x\in \Lambda^{\infty}|\; \lambda=x_{1}x_{2}\cdots x_{d(\lambda)} \}$$
parametrised by finite length paths $\lambda$ form a basis of open set in the subspace topology. 

\begin{remark}[\cite{kumjian-pask}]
    Generally speaking, infinite path space for higher rank graph is given by the set of all degree-preserving morphisms from $\Omega_{k}$ to $\Lambda$, where $\Lambda$ is a $k$-graph. However, for a $1$-graph, the situation is simpler. Using the \textit{unique factorisation property}, given a degree preserving morphism $x:\Omega_{1}\rightarrow \Lambda$, we can uniquely identify $x$ with an infinite path $$\lambda=x(0,1)x(1,2)x(2,3)x(3,4)\cdots$$ 
 vice versa. As a result, the typical endless path and the infinite path provided by the morphism for the $1$-graph may be identified with each other.
\end{remark} 
For the following proposition, see \cite{farsi}.
\begin{proposition}
\label{unique KMS state}
    Let $\Lambda$ be a strongly connected finite $1$-graph with the vertex matrix $A$. There is a unique Borel probability measure $M$ on $\Lambda^{\infty}$, given on the cylinder sets by 
\begin{equation*}
    M([\lambda])=\rho(\Lambda)^{-d(\lambda)}x_{s(\lambda)}^{\Lambda}
\end{equation*}
Where $\rho(\lambda)$ is the spectral radius of the vertex matrix $A$ and $(x_{v}^{\Lambda})_{v\in V}$ is the unique normalized Perron-Frobenious eigenvector of $\Lambda$.
\end{proposition}
%  The theory of semi-branching functions provides a representaion of higher rank graph algebras on seperable representation Hilbert space $L^{2}(\Lambda^{\infty},M)$; faithful for strongly connected finite aperiodic $k$-graphs. 
% Using the semi-branching function, later \cite{Farsi, et.al} constructed a spectral triple for strongly connected higher rank graph which generalize the construction of \cite{Marcolli} on Cuntz-Kriger algebra $\mathcal{O_{A}}$
% Moreover the wavelet decomposition of $\mathcal{H}=L^{2}(\Lambda^{\infty},M)$ coincide with eigen space of Dirac operator.
Now we are ready to introduce the spectral triple given by Farsi et al. in \cite{farsi}. Recall that a spectral triple is a data comprising of a $C^{\ast}$-algebra, a representation of the algebra on a Hilbert space and a Dirac operator. We are going to describe the data next. To that end, let $\Lambda=(V,E,r,s)$ be a strongly connected finite graph without multiple edges, loops and sources. Let the corresponding $1$-graph be $\Lambda$.

\indent \textbf{The Hilbert space}: The Hilbert space of the spectral triple is taken to be $L^{2}(\Lambda^{\infty},M)$. There is a very useful dense subspace of the Hilbert space given by the linear span of the characteristic functions of the cylinder sets parametrized by finite length paths. 
%Higher rank graph algebras are represented on the separable Hilbert space $L^{2}(\Lambda^{\infty},M)$ using the theory of semi-branching functions; this representation is faithful for strongly connected finite aperiodic $k$-graphs. 
%Later, \cite{farsi} constructed a spectral triple for strongly connected finite higher rank graphs using the semi-branching function, which generalizes the construction of \cite{marcolli} spectral triple on Cuntz-Kriger algebras $\mathcal{O_{A}}$.
%Additionally, the wavelet decomposition of $\mathcal{H}=L^{2}(\Lambda^{\infty},M)$ coincides with the eigen spaces of Dirac operator.

\textbf{The representation of $C^{\ast}(\Lambda)$}: Consider the seperable Hilbert space $\mathcal{H}=L^{2}(\Lambda^{\infty},M)$. Using a semi-branching function system on the measure space $(\Lambda^{\infty},M)$, one can construct a representation $\pi: C^{\ast}(\Lambda)\raro B(\mathcal{H})$ which is given by the following: 
\begin{equation}
\label{rep-1}
    \pi(S_{\lambda})(\chi_{[\eta]})= \rho(\Lambda)^{d(\lambda)/2} \chi_{[\lambda\eta]},
\end{equation}
where $\lambda,\eta$ are finite paths and $\chi_{[\eta]}$ is a characteristic function on the cylinder set $[\eta]$ as usual. To define $\pi(S_{\lambda}^{\ast})$, we consider two cases: \\
{\bf Case I} $d(\lambda)\geq d(\eta)$: There are two subcases. Firstly, let $\lambda=\eta\beta$ for some $\beta$. Then 
\begin{equation}
\label{rep-2}
    (\pi(S_{\lambda}))^{\ast}(\chi_{[\eta]})=\rho(\Lambda)^{-d(\lambda)/2}\chi_{[s(\lambda)]}
\end{equation}
Otherwise 
\begin{equation}
\label{rep-2-1}
    (\pi(S_{\lambda}^{\ast}))(\chi_{[\eta]})=0.
\end{equation}
{\bf Case 2}:  $d(\lambda)<d(\eta)$: Again there are two subcases. Firstly, let $\lambda\beta=\eta$ for some $\beta$. Then 
\begin{equation}
\label{rep-3}
    (\pi(S_{\lambda}^{\ast}))(\chi_{[\eta]})=\rho(\Lambda)^{-d(\lambda)/2}\chi_{[\beta]}
\end{equation}
Otherwise as before,
\begin{equation}
\label{rep-3-1}
    (\pi(S_{\lambda}^{\ast}))(\chi_{[\eta]})=0.
\end{equation}
%For strongly connected finite graph this representation is faithful if and only if the corresponding 1-graph is aperiodic and 1-graph for the directed graph $\Lambda$ is aperiodic if and only if every loop has an exit, hence for our case of interest this representation is faithful.\\

\textbf{The Dirac Operator}: Define $\mathcal{R}_{-1}$ to be the linear subspace of $L^{2}(\Lambda^{\infty},M)$ spanned by constant function on $\Lambda^{\infty}$. For $k\in \mathbb{N}$, define $\mathcal{R}_{k}\subset L^{2}(\Lambda^{\infty},M) $ by, 
\begin{equation*}
    \mathcal{R}_{k}=\text{Sp}\{\chi_{[\eta]} |\; \eta\in \Lambda^{\ast}, d(\eta)\leq k\}
\end{equation*}
%\textbf{Linear basis and dimension of $\mathcal{R}_{k}$}:
Now let us note down the following proposition which is going to be important in proving the main results of this paper.
\begin{proposition}
\label{R_kdim}
  For a fixed $k>1$, consider the subset $\mathcal{C}_{k}=\{\chi_{[\eta]}|\; d(\eta)=k\}\subset \mathcal{R}_{k}$. Then elements of $\mathcal{C}_{k}$ forms a linear basis for the space $\clr_{k}$. Consequently, if the graph has $m$ number of edges, ${\rm dim}(\clr_{k})\leq m^{k}$. 
\end{proposition}
\begin{proof}
First we can observe that for any $\lambda\in \Lambda^{\ast}$ and $n\in \mathbb{N}$,
$$[\lambda]=\bigcup \limits_{\substack{\mu\in \Lambda^{\ast}, d(\mu)=n\\ s(\mu)=r(\lambda)}}[\mu\lambda]$$ 
For $\eta_{1},\eta_{2}\in\Lambda^{\ast}$ such that $\eta_{1}\neq\eta_{2}$ and $d(\eta_{1})=d(\eta_{2})$, it is easy to see that $[\eta_{1}]\cap[\eta_{2}]=\emptyset$. Thus for $\mu_{1}\neq\mu_{2}$ with $d(\mu_{1})=d(\mu_{2})$, $[\mu_{1}\lambda]\cap[\mu_{2}\lambda]=\emptyset$. This implies that any $x\in[\lambda]$ lies in a unique cylinder set $[\mu\lambda]$ for a unique finite path $\mu$ with $d(\mu)=n$. Therefore, 
\begin{equation*}
\chi_{[\lambda]}=\sum\limits_{\substack{\mu,d(\mu)=n\\ s(\mu)=r(\lambda)}}\chi_{[\mu\lambda]}
\end{equation*}
Now if we consider any $\lambda\in \Lambda^{\ast}$ such that $d(\lambda)<k$ then we have, 
\begin{equation*}
\chi_{[\lambda]}=\sum\limits_{\substack{\mu,d(\mu)=k-d(\lambda)\\ s(\mu)=r(\lambda)}}\chi_{[\mu\lambda]}
\end{equation*}
As $d(\mu\lambda)=k$ in the above formula, every element of $\mathcal{R}_{k}$ is a linear combination of elements in $\mathcal{C}_{k}$
. Now let \begin{displaymath}\sum\limits_{\eta,d(\eta)=k}c_{\eta}\chi_{[\eta]}=0,\end{displaymath} for some scalars $c_{\eta}$. As the intersection of any two cylinder sets is empty for different paths of same degree, by choosing one of the infinite paths for each $\eta$ we get $c_{\eta}=0$. Hence these vectors are linearly independent in $\mathcal{R}_{k}$. There can be at most $m^{k}$ paths of length $k$ if number of edges in the graph is $m$ giving the desired bound on the dimension of $\clr_{k}$. \end{proof}
Let $\Xi_{k}$ be the orthogonal projection in $L^{2}(\Lambda^{\infty},M)$ onto $\mathcal{R}_{k}$. For a pair $(k,l)\in \mathbb{N}\times (\mathbb{N}\cup \{-1\})$ with $k>l$, let 
\begin{equation*}
    \widehat{\Xi}_{k,l}=\Xi_{k}-\Xi_{l}.
\end{equation*}
As $\mathcal{R}_{l}\subset \mathcal{R}_{k}$ and hence $\widehat{\Xi}_{k,l}$ is an orthogonal projection onto $\mathcal{R}_{k}\cap \mathcal{R}_{l}^{\bot}$.
Given an increasing sequence of $\alpha=(\alpha_{q})$ of positive real number with $\text{lim}_{q\raro \infty}\alpha_{q}=\infty$ such that there exists a $C>0$ for which $|\alpha_{q+1}-\alpha_{q}|<C$ for all $q$, define an unbounded operator $D$ on $L^{2}(\Lambda^{\infty},M)$ by, 
\begin{equation}
\label{diracoperatorformula}
    D:= \sum\limits_{q\in \mathbb{N}}\alpha_{q}\widehat{\Xi}_{q,q-1}. 
\end{equation}
For the proof of the following proposition, the reader is referred to \cite{farsi}.
\begin{proposition}
\label{spectral triple}
    Let $\Lambda=(V,E,r,s)$ be a finite strongly connected directed graph without multiple edges and loops and the corresponding $1$-graph be $\Lambda$. Let $\mathcal{A}_{\Lambda}$ be the dense $\ast$-subalgebra of $C^{\ast}(\Lambda)$ spanned by $S_{\mu}S_{\eta}^{\ast}$ and $D$ be the Dirac operator on $\mathcal{H}=L^{2}(\Lambda^{\infty},M)$ given by the formula \ref{diracoperatorformula}. Then the data $(\mathcal{A}_{\Lambda}, \mathcal{H}, D)$ gives a spectral triple for $C^{\ast}(\Lambda)$.
\end{proposition}
Now we prove the following proposition which is a bit of digression. This proposition is not going to be used in proving the main results in this paper.
\begin{proposition}
    With the choice $\alpha_{q}=q^{\frac{1}{2}+\epsilon}$ for some $0<\epsilon<\frac{1}{2}$ in the formula \ref{diracoperatorformula} of the Dirac operator, $(C^{\ast}(\Lambda),L^{2}(\Lambda^{\infty},M),D)$ is a $\Theta$-summable spectral triple.
\end{proposition}
\begin{proof}
 It is easy to see that $|\alpha_{q+1}-\alpha_{q}|\leq 1$ for all $q$ so that $(C^{\ast}(\Lambda),L^{2}(\Lambda^{\infty},M),\cld)$ is a spectral triple. As $e^{-tD^2}$ is a positive operator for all $t>0$, to prove that $e^{-tD^2}$ is trace class, it is enough to show that ${\rm Tr}(e^{-tD^2})<\infty$ for all $t>0$. If ${\rm dim}(\clr_{q}\cap \clr_{q-1}^{\perp})$ is $n_{q}$, then fixing an orthonormal basis $\{\xi_{1},\ldots,\xi_{n_{q}}\}$ of $(\clr_{q}\cap \clr_{q-1}^{\perp})$, we get
 \begin{equation*}
     {\rm Tr}(e^{-tD^2})=\sum_{q\in\mathbb{N},j=1}^{j=n_{q}}\langle e^{-tD^2}\xi_{j},\xi_{j}\rangle\\
    = \sum_{q\in\mathbb{N}}e^{-tq^{1+2\epsilon}}n_{q}.
 \end{equation*}
 Now if we assume that the number of edges of the graph is $m$, then by Proposition \ref{R_kdim}, $n_{q}\leq m^q$. Then the above series is dominated by the series $\sum_{q}e^{-tq^{1+2\epsilon}}m^{q}$. If we denote the $q$-th term of the later by $a_{q}$, then $a_{q}>0$ for all $q$ and 
 \begin{equation*}
     {\rm lim}_{q}(a_{q})^{1/q}={\rm lim}_{q}m(e^{-tq^{1+2\epsilon}})^{\frac{1}{q}}={\rm lim}_{q}m/(e^{tq^{1+2\epsilon}})^{\frac{1}{q}}.
 \end{equation*} Now the denominator goes to infinity for all $t>0$ as $q\raro\infty$. This can be seen by taking logarithm of the denominator. Therefore, ${\rm lim}_{q}(a_{q})^{1/q}=0$. Hence by the root test, ${\rm Tr}(e^{-tD^2})<\infty$ for all $t>0$.\end{proof}

\subsubsection{Isometric Action}
Now we are going to recall the notion of `orientation preserving quantum isometric action'. We shall drop the term orientation preserving in the sequel. The main ideas are taken from \cite{goswami1}. We do slight modifications according to our needs. We start with the standard definition of unitary co-representation of a CQG on a Hilbert space.
\begin{definition}(see \cites{vandaele,neshveyevrepcategory})
\label{Unitary representation definition}
    A \textit{unitary co-representation} of a CQG $\mathbb{G}$ on a Hilbert space $\mathcal{H}$ (\textit{possibly infinite dimensional}) is a unitary element $\widetilde{U}$ of $\mathcal{M}(\mathcal{K}(\mathcal{H})\otimes C(\mathbb{G}))$
such that 
\begin{displaymath}({\rm id}\otimes \Dlt)\circ\widetilde{U}=\widetilde{U}_{(12)}\widetilde{U}_{(13)},\end{displaymath}
\end{definition}
There is a standard equivalent picture of a unitary co-representations which we recall now. It is going to be useful in the sequel. Let $U$ be a $\mathbb{C}$-linear map from a Hilbert space $\mathcal{H}$ to the Hilbert $C(\mathbb{G})$-module $\mathcal{H}\overline{\otimes} C(\mathbb{G})$ which satisfies
\begin{enumerate}[(I)]
\setlength\itemsep{0.5em}
    \item $\langle U(\xi), U(\eta) \rangle=\langle \xi,\eta \rangle.1$ for all $\xi,\eta\in\clh$;
    \item $(U\otimes \rm{id})\circ U=({\rm id}\otimes \Dlt)\circ U$ on some dense subspace $\clk\subset\clh$ where $U$ is algebraic i.e. $U(\clk)\subset\clk\otimes_{\rm alg}C(\mathbb{G})_{0}$;
    \item ${\rm Sp}\{U(\xi).q\;|\; \xi\in \mathcal{H}, q\in C(\mathbb{G})\}$ is dense in $\mathcal{H}\overline{\otimes} C(\mathbb{G})$.
\end{enumerate}
Given such a $\mathbb{C}$-linear map, we have a unitary adjointable map on the Hilbert module $\mathcal{H}\overline{\otimes}C(\mathbb{G})$ given by $\widetilde{U}(\xi\otimes q)=U(\xi).q$. Therefore, we have $\widetilde{U}\in \cll(\mathcal{H}\overline{\otimes}C(\mathbb{G}))=\mathcal{M}(\mathcal{K}(\mathcal{H})\otimes C(\mathbb{G}))$. Moreover $\widetilde{U}$ satisfies $({\rm id}\otimes \Dlt)\circ\widetilde{U}=\widetilde{U}_{(12)}\widetilde{U}_{(13)}$ making $\widetilde{U}$ a unitary co-representation of $\mathbb{G}$ on $\clh$. Conversely, given a unitary co-representation $\widetilde{U}$, one can construct the $\mathbb{C}$-linear map $U$ by the formula $U(\xi):=\widetilde{U}(\xi\otimes 1)$. Then it is easy to check that this $U$ satisfies (I), (II) and (III).\\
\indent Now let $(\cla,\clh,\cld)$ be a spectral triple on a $C^{\ast}$-algebra and $\alpha:\cla\raro\cla\otimes C(\mathbb{G})$ be an action of a CQG $\mathbb{G}$ on $\cla$. Also let $\tau$ be a faithful state on $C(\mathbb{G})$ with $L^{2}(\mathbb{G})$ being the GNS-space with respect to $\tau$. If we denote the representation of $\cla$ on $\clh$ by $\pi$ and the GNS representation of the $C^{\ast}$-algebra $C(\mathbb{G})$ on $L^{2}(\mathbb{G})$ by $[.]$, then there is a representation of the $C^{\ast}$-algebra $\cla\otimes C(\mathbb{G})$ on $\clb(\clh\otimes L^{2}(\mathbb{G}))$. We denote the representation by $(\pi\otimes [.])$. We say that the action $\alpha$ is implemented on the spectral data by some unitary co-representation $\widetilde{U}$ of $\mathbb{G}$ on $\clh$ if 
\begin{displaymath}
    (\pi\otimes [.])\alpha(a)=\widetilde{U}(\pi(a)\otimes [1])\widetilde{U}^{\ast}\in\clb(\clh\otimes L^{2}(\mathbb{G})) \ \forall \ a\in\cla,
\end{displaymath}
where $1$ is the identity element of $C(\mathbb{G})$. Note that $\clm(\clk(\clh)\otimes C(\mathbb{G}))$ is a subalgebra of $\clb(\clh\otimes L^{2}(\mathbb{G}))$ and therefore $\widetilde{U}\in \clb(\clh\otimes L^{2}(\mathbb{G}))$. We shall denote the map $\widetilde{U}(\pi(.)\otimes[.])\widetilde{U}^{\ast}$ from $\cla$ to $B(\clh\otimes L^{2}(\mathbb{G}))$ by ${\rm ad}_{\widetilde{U}}$. Clearly ${\rm ad}_{\widetilde{U}}$ is a $C^{\ast}$-homomorphism. The reader can compare the following definition with Definition 2.6 of \cite{goswami1} (also see the discussion in the second paragraph on page no. 1803 of \cite{rossi}).
\begin{definition}
\label{quantumisometricdefinition}
    Let $(\cla,\clh,D)$ be a spectral triple where $\cla$ is a unital $C^{\ast}$-algebra. An action $\alpha$ of a CQG $\mathbb{G}$ on $\cla$ is said to be `quantum isomteric' (or simply isometric) if 
    \begin{enumerate}[(I)]
    \setlength\itemsep{0.5em}
    \item $\alpha$ is implemented on the spectral data by some unitary co-representation $\widetilde{U}$.
    \item $\widetilde{U}\circ(D\otimes 1)=(D\otimes 1)\circ\widetilde{U}$ on appropriate domains.
    \end{enumerate} 
\end{definition}
%\begin{definition}
 %   A compact quantum group $(C(\mathbb{G}),\Dlt)$ is said to be act \textit{isometrically} on $(\mathcal{A},\mathcal{H},D)$, if there exists a unitary co-representations $\widetilde{U}\in \mathcal{M}(\mathcal{K}(\mathcal{H})\otimes C(\mathbb{G}))$ which satisfies,
%\begin{enumerate}[(I)]
%\setlength\itemsep{0.5em}
 %   \item $\widetilde{U}(D\otimes \rm{id})=(D\otimes \rm{id})\widetilde{U}$
  %  \item 
%\end{enumerate}
%\end{definition}
% \begin{remark}
%     The spectral triple constructed by Marcolli and Consani is slightly different as they have use 
% \end{remark}
\begin{remark}
    Note that although in the definition of a spectral triple $(\cla,\clh,\cld)$, it is in general assumed that the representation $\pi:\cla\raro\clb(\clh)$ is faithful, our consideration of quantum isometric action does not really require $\pi$ to be faithful. In this context we would like to remind the reader that for the spectral triple $(C^{\ast}(\Lambda),L^{2}(\Lambda^{\infty},M),\cld)$, the representation $\pi$ is faithful if and only if the corresponding $1$-graph is aperiodic if and only if every closed cycle has an exit. 
\end{remark}
\section{Main section}
In this section we are going to prove the following main theorem:
\begin{theorem}
    \label{maintheorem}
    Let $\Lambda$ be a strongly connected $1$-graph without multiple edge, loop or source and $(C^{\ast}(\Lambda),L^{2}(\Lambda^{\infty},M),\cld)$ be the spectral triple as given in Proposition \ref{spectral triple}. Then the action of ${\rm Aut}^{+}(\Lambda)$ on $C^{\ast}(\Lambda)$ given in Proposition \ref{actionofqaut} is isometric in the sense of Definition \ref{quantumisometricdefinition}.
\end{theorem}We start with constructing a unitary co-representation $\widetilde{U}$ of $Aut^{+}(\Lambda)$ on the Hilbert space $\clh=L^{2}(\Lambda^{\infty},M)$. We are going to denote the Hilbert space $L^{2}(\Lambda^{\infty},M)$ by $\clh$ throughout this section. Let us denote the dense subspace say of $\clh$ given by the increasing union $\cup_{k}\clr_{k}$ by $\clk$. The idea is to define an isometric $\mathbb{C}$-linear map $U$ from $\clk$ to $\clh\overline{\otimes}C(\mathbb{G})$ and extending it to the whole of $\clh$. The important observation is that since we are seeking for a unitary co-representation that would implement the given action $\alpha$, once the representation is defined on $\clr_{0}$, it is essentially determined by the action $\alpha$ on the whole of $\clh$ which leads to the natural formula for $U$ on the eigenspaces $\clr_{k}$ (see Formula \ref{Uformula k lenghth path} below). The challenging part is to prove that the formulae gel together to produce a well-defined unitary $\widetilde{U}\in\clm(\clk(\clh)\otimes C(\mathbb{G}))$. With this heuristics we start with the following theorem.    
\begin{theorem}
\label{U-map}
     There exists a well-defined linear $\mathbb{C}$-linear map $U:\mathcal{K}\raro \mathcal{K}\otimes_{{\rm alg}} C(\rm{Aut}^{+}(\Lambda))_{0}$ given on the characteristic function of a cylinder set $[\lambda]$ for some path $\lambda=\lambda_{1}\lambda_{2}\ldots\lambda_{k}$ of length $k$ by the following formula:
     \begin{equation}
\label{Uformula k lenghth path}
    U(\chi_{[\lambda]})=\sum_{\eta}\chi_{[\eta]}\otimes q_{r(\eta_{1})r(\lambda_{1})}q_{s(\eta_{1})s(\lambda_{1})}\cdots q_{r(\eta_{k})r(\lambda_{k})}q_{s(\eta_{k})s(\lambda_{k})}
\end{equation}
where the summation varies over all the paths $\eta=\eta_{1}\eta_{2}\ldots\eta_{k}$ of degree $k$. Moreover, $(U\otimes{\rm id})\circ U=({\rm id}\otimes\Delta)\circ U$ on $\clk$.
\end{theorem}
\begin{proof}
    Fix $k\in \mathbb{N}$ and let $\mathcal{R}_{k}$ be the finite-dimensional subspace of $L^2(\Lambda^{\infty},M)$ for which $\mathcal{C}_{k}$ forms a basis of $\mathcal{R}_{k}$ as proved in proposition \ref{R_kdim}. So we have a well-defined $\mathbb{C}$-linear map $U_{k}:\mathcal{R}_{k}\raro \mathcal{K}\otimes_{{\rm alg}}C(\text{Aut}^{+}(\Lambda))_{0}$, given on the basis elements by formula \ref{Uformula k lenghth path}.  
  Note that for $k=0$, $\clr_{0}$ is the vector space ${\rm Sp}\{\chi_{[i]}:i=1,2,\ldots,n\}$, where the vertices of $\Lambda$ are $\{1,2,\ldots,n\}$. Then the map $U_{0}$ on $\clr_{0}$ has the following formula:
\begin{equation}
    \label{Uformula vertices}
    U_{0}(\chi_{[i]})=\sum\limits_{j}\chi_{[j]}\otimes q_{ji}.
\end{equation}
As $\mathcal{R}_{k}$ forms an increasing chain of finite-dimensional subspaces, to get a well defined linear map on $\mathcal{K}$, we just need to show that for $l<k$ the maps satisfy, $U_{k}|_{\mathcal{R}_{l}}=U_{l}$. To show this we need to show equality only for the basis set $\mathcal{C}_{l}$ of $\clr_{l}$. Let $\chi_{[\lambda]}\in \mathcal{C}_{l}$. Then an element in $\mathcal{R}_{k}$ can written as a linear combination of basis element of $\mathcal{R}_{k}$. In particular, 
\begin{equation*}
\chi_{[\lambda]}=\sum\limits_{\substack{\mu,d(\mu)=k-l\\ s(\mu)=r(\lambda)}}\chi_{[\mu\lambda]}
\end{equation*}
where $\mu$ varies over all the paths of length $m=k-l$. So we need to prove, 
\begin{equation*}
    U_{l}(\chi_{[\lambda]})=\sum\limits_{\substack{\mu,d(\mu)=m\\ s(\mu)=r(\lambda)}}U_{k}(\chi_{[\mu\lambda]})
\end{equation*}
Expanding LHS by definition we get, 
\begin{equation*}
U(\chi_{[\lambda]})=\sum\limits_{\substack{\xi\in \Lambda^{*}\\
d(\xi)=l}}\chi_{[\xi]}\otimes q_{r(\xi_{m+1})r(\lambda_{m+1})}q_{s(\xi_{m+1})s(\lambda_{m+1})}\cdots q_{r(\xi_{m+l})r(\lambda_{m+l})}q_{s(\xi_{m+l})s(\lambda_{m+l})}
\end{equation*}
where we are writing $\xi=\xi_{m+1}\xi_{m+2}\cdots\xi_{m+l}$ and $\lambda=\lambda_{m+1}\lambda_{m+2}\cdots\lambda_{m+l}$.\\
On the other hand by expanding RHS we get, 
% Testing equation
% better way 
\begin{equation*}
    \sum\limits_{\substack{\mu\in \Lambda^{\ast}\\ d(\mu)=m, \\s(\mu)=r(\lambda)}}U(\chi_{[\mu\lambda]}) =\sum\limits_{\substack{\mu\in \Lambda^{\ast} \\ d(\mu)=m}}\sum\limits_{\substack{\delta\in \Lambda^{\ast} \\ d(\delta)=k}}\chi_{[\delta]}\otimes
    \begin{aligned}[t]
        &q_{r(\delta_{1})r(\mu_{1})}q_{s(\delta_{1})s(\mu_{1})}\cdots q_{r(\delta_{m})r(\mu_{m})}q_{s(\delta_{m})s(\mu_{m})} \\
        &q_{r(\delta_{m+1})r(\lambda_{m+1})}q_{s(\delta_{m+1})s(\lambda_{m+1})}\cdots\cdots \\
        &\cdots\cdots
q_{r(\delta_{m+l})r(\lambda_{m+l})}q_{s(\delta_{m+l})s(\lambda_{m+l})}
    \end{aligned}
\end{equation*}
Where $\delta=\delta_{1}\cdots\delta_{m}\delta_{m+1}\cdots\delta_{m+l}$ and $\mu=\mu_{1}\mu_{2}\cdots\mu_{m}$.\\
Since every path $\delta$ of degree $k=m+l$ has a unique decomposition into a path $\zeta,\xi$ of degree $m$ and $l$, where we write $\xi=\xi_{m+1}\cdots\xi_{m+l},\zeta=\zeta_{1}\zeta_{2}\cdots\zeta_{m}$ and the above summation becomes,
% checking
\begin{equation*}
    =\sum\limits_{\substack{\mu\in \Lambda^{\ast}, d(\mu)=m \\ s(\mu)=r(\lambda)}} \sum\limits_{\substack{\delta=\zeta\xi\in \Lambda^{\ast} \\ d(\zeta)=m, d(\xi)=l}}\chi_{[\zeta\xi]}\otimes 
    \begin{aligned}[t]
        &q_{r(\zeta_{1})r(\mu_{1})}q_{s(\zeta_{1})s(\mu_{1})}\cdots q_{r(\zeta_{m})r(\mu_{m})}q_{s(\zeta_{m})s(\mu_{m})} \\
        &q_{r(\xi_{m+1})r(\lambda_{m+1})}q_{s(\xi_{m+1})s(\lambda_{m+1})}\cdots\cdots\cdots \\
        &\cdots\cdots q_{r(\xi_{m+l})r(\lambda_{m+l})}q_{s(\xi_{m+l})s(\lambda_{m+l})}
    \end{aligned}
\end{equation*}
\begin{equation*}
    =\sum\limits_{\substack{\mu\in \Lambda^{\ast} \\ s(\mu)=r(\lambda)}} \sum\limits_{\substack{\xi\in \Lambda^{\ast}\\ d(\xi)=l }}\sum\limits_{\substack{\zeta\in \Lambda^{\ast}, d(\zeta)=m \\ s(\zeta)=r(\xi)}} \chi_{[\zeta\xi]}\otimes
    \begin{aligned}[t]
        &q_{r(\zeta_{1})r(\mu_{1})}q_{s(\zeta_{1})s(\mu_{1})}\cdots q_{r(\zeta_{m})r(\mu_{m})}q_{s(\zeta_{m})s(\mu_{m})}\\
        &q_{r(\xi_{m+1})r(\lambda_{m+1})}q_{s(\xi_{m+1})s(\lambda_{m+1})}\cdots\cdots\\
        &\cdots\cdots q_{r(\xi_{m+l})r(\lambda_{m+l})}q_{s(\xi_{m+l})s(\lambda_{m+l})}
    \end{aligned}
\end{equation*}
We shall prove now that the coefficients for a fixed $\xi$ and for any $\zeta$ with $s(\zeta)=r(\xi)$ are the  same. Now in the summands for fix $\xi,\zeta$ such that $s(\zeta)=r(\xi)$, the coefficients of $\chi_{[\zeta\xi]}$ is
\begin{equation*}
    =\sum\limits_{\substack{\mu, d(\mu)=m \\ s(\mu)=r(\lambda)}}
    \begin{aligned}[t]
        &\;q_{r(\zeta_{1})r(\mu_{1})}q_{s(\zeta_{1})s(\mu_{1})}\cdots q_{r(\zeta_{m})r(\mu_{m})}q_{s(\zeta_{m})s(\mu_{m})}
q_{r(\xi_{m+1})r(\lambda_{m+1})}q_{s(\xi_{m+1})s(\lambda_{m+1})}\\
&\; q_{r(\xi_{m+2})r(\lambda_{m+2})}q_{s(\xi_{m+2})s(\lambda_{m+2})} \cdots q_{r(\xi_{m+l})r(\lambda_{m+l})}q_{s(\xi_{m+l})s(\lambda_{m+l})}
    \end{aligned}
\end{equation*}
Now note that in the above summation, $(r(\mu_{1}),s(\mu_{1}))$ can be replaced by $(i,j)\in V\times V$. This can be proved by the following argument: Whenever $(i,j)\notin E$, $q_{r(\zeta_{1})i}q_{s(\zeta_{1}),j}=0$ by \ref{qaut_equ}. As for any $(i,j)\in E$, there are two possibilities: Either $(i,j)$ is the first edge of a $k$-length path in which case it is already included in the summation, or $(i,j)$ fails to be the first edge of a $k$-length path. In this case, as $\zeta_{1}\cdots\zeta_{m}\xi_{m+1}\cdots\xi_{m+l}$ forms a $k$-length path,
\begin{displaymath}
    q_{r(\zeta_{1})i}q_{s(\zeta_{1})j}\cdots q_{r(\zeta_{m})r(\mu_{m})}q_{s(\zeta_{m})s(\mu_{m})}
q_{r(\xi_{m+1})r(\lambda_{m+1})}\cdots q_{s(\xi_{m+l})s(\lambda_{m+l})}=0,
\end{displaymath} which enables us to replace $(r(\mu_{1}),s(\mu_{1}))$ by $(i,j)\in V$. Therefore the summand becomes equal to
\begin{equation*}
    \sum\limits_{\substack{\mu, d(\mu)=m-1 \\ s(\mu)=t(\lambda)}}\sum\limits_{i,j\in V}
    \begin{aligned}[t]
        &q_{r(\zeta_{1})i}q_{s(\zeta_{1})j}\cdots q_{r(\zeta_{m})r(\mu_{m})}q_{s(\zeta_{m})s(\mu_{m})}
q_{r(\xi_{m+1})r(\lambda_{m+1})}q_{s(\xi_{m+1})s(\lambda_{m+1})} \\
&q_{r(\xi_{m+2})r(\lambda_{m+2})}q_{s(\xi_{m+2})s(\lambda_{m+2})} \cdots q_{r(\xi_{m+l})r(\lambda_{m+l})}q_{s(\xi_{m+l})s(\lambda_{m+l})}
    \end{aligned}
\end{equation*}
\begin{equation*}
    =\sum\limits_{\substack{\mu, d(\mu)=m-1 \\ s(\mu)=r(\lambda) \\ \mu=\mu_{2}\cdots\mu_{m}}}
    \begin{aligned}[t]
        &q_{r(\zeta_{2})r(\mu_{2})}q_{s(\zeta_{2})s(\mu_{2})}\cdots q_{r(\zeta_{m})r(\mu_{m})}q_{s(\zeta_{m})s(\mu_{m})}q_{r(\xi_{m+1})r(\lambda_{m+1})}q_{s(\xi_{m+1})s(\lambda_{m+1})} \\
        &q_{r(\xi_{m+2})r(\lambda_{m+2})}q_{s(\xi_{m+2})s(\lambda_{m+2})}\cdots  q_{r(\xi_{m+l})r(\lambda_{m+l})}q_{s(\xi_{m+l})s(\lambda_{m+l})}
    \end{aligned}
\end{equation*}
By repeating the same argument, we can get rid of the first $m$ indices and the summand reduces to 
\[
q_{r(\xi_{m+1})r(\lambda_{m+1})}q_{s(\xi_{m+1})s(\lambda_{m+1})} \cdots q_{r(\xi_{m+l})r(\lambda_{m+l})}q_{s(\xi_{m+l})s(\lambda_{m+l})}
\]
Therefore the term $ \sum\limits_{\substack{\mu\in \Lambda^{\ast}, d(\mu)=m\\
s(\mu)=r(\lambda)}}U(\chi_{[\mu\lambda]})$ becomes
\begin{equation*}
    =\sum\limits_{\substack{\xi\in \Lambda^{\ast}, d(\xi)=l \\ \xi=\xi_{m+1}\cdots\xi_{m+l}}}\sum\limits_{\substack{\zeta\in \Lambda^{\ast} , d(\zeta)=m\\ \zeta=\zeta_{1}\zeta_{2}\cdots\zeta_{m}, \\ s(\zeta)=r(\xi)}}\chi_{[\zeta\xi]} \otimes \sum\limits_{\substack{\mu\in \Lambda^{\ast} \\ d(\mu)=m\\ s(\mu)=r(\lambda)}} 
    \begin{aligned}[t]
        &q_{r(\zeta_{1})r(\mu_{1})}q_{s(\zeta_{1})s(\mu_{1})}\cdots q_{r(\zeta_{m})r(\mu_{m})}q_{s(\zeta_{m})s(\mu_{m})}\\
        &q_{r(\xi_{m+1})r(\lambda_{m+1})}q_{s(\xi_{m+1})s(\lambda_{m+1})}\cdots\cdots\\
        &\cdots\cdots
        q_{r(\xi_{m+l})r(\lambda_{m+l})}q_{s(\xi_{m+l})s(\lambda_{m+l})}
    \end{aligned}
\end{equation*}
which by the above calculations equals to
\begin{equation*}
    \sum\limits_{\substack{\xi\in \Lambda, d(\xi)=l \\ \xi=\xi_{m+1}\cdots\xi_{m+l}}}\sum\limits_{\substack{\zeta, d(\zeta)=m\\ \zeta=\zeta_{1}\zeta_{2}\cdots\zeta_{m}, \\ s(\zeta)=r(\xi)}}\chi_{[\zeta\xi]}\otimes
    \begin{aligned}[t]
&q_{r(\xi_{m+1})r(\lambda_{m+1})}q_{s(\xi_{m+1})s(\lambda_{m+1})}q_{r(\xi_{m+2})r(\lambda_{m+2})} \\
&q_{s(\xi_{m+2})s(\lambda_{m+2})}\cdots\cdots q_{r(\xi_{m+l})r(\lambda_{m+l})}q_{s(\xi_{m+l})s(\lambda_{m+l})}
    \end{aligned}
\end{equation*}
\begin{equation*}
    =\sum\limits_{\substack{\xi\in \Lambda, d(\xi)=l \\ \xi=\xi_{m+1}\cdots\xi_{m+l}}}\left(\sum\limits_{\substack{\zeta, d(\zeta)=m\\ \zeta=\zeta_{1}\zeta_{2}\cdots\zeta_{m}, s(\zeta)=r(\xi)}}\chi_{[\zeta\xi]}\right)\otimes 
    \begin{aligned}[t]
        &q_{r(\xi_{m+1})r(\lambda_{m+1})}q_{s(\xi_{m+1})s(\lambda_{m+1})}\cdots\cdots \\
        &\cdots
        \cdots q_{r(\xi_{m+l})r(\lambda_{m+l})}q_{s(\xi_{m+l})s(\lambda_{m+l})}
    \end{aligned}
\end{equation*}
As we know that from the proof of proposition \ref{R_kdim}, $\chi_{[\xi]}=\sum\limits_{\substack{\zeta, d(\zeta)=m\\ \zeta=\zeta_{1}\zeta_{2}\cdots\zeta_{m}\\ s(\zeta)=r(\xi)}}\chi_{[\zeta\xi]}$ and therefore the summation becomes
\begin{equation*}
    \sum\limits_{\substack{\xi\in \Lambda^{*}\\ d(\xi)=l \\ \xi=\xi_{m+1}\cdots\xi_{m+l}}}\chi_{[\xi]}\otimes q_{r(\xi_{m+1})r(\lambda_{m+1})}q_{s(\xi_{m+1})s(\lambda_{m+1})}\cdots q_{r(\xi_{m+l})r(\lambda_{m+l})}q_{s(\xi_{m+l})s(\lambda_{m+l})}
\end{equation*}
which is nothing but $U(\chi_{[\lambda]})$.
In conclusion we have,
\begin{equation*}
U_{l}(\chi_{[\lambda]})=\sum\limits_{\substack{\mu,d(\mu)=m\\ s(\mu)=r(\lambda)}}U_{k}(\chi_{[\mu\lambda]}) 
\end{equation*}
and we have a well defined $\mathbb{C}$-linear map $U:\mathcal{K}\raro \mathcal{K}\otimes_{{\rm alg}} C(\text{Aut}^{+}(\Lambda))_{0}$ where $\mathcal{K}=\cup_{k}\mathcal{R}_{k}$ is dense in $L^2(\Lambda^{\infty},M)$. The formula $(U\otimes{\rm id})\circ U=({\rm id}\otimes \Delta)\circ U$ follows from the well known fact that the action $\alpha$ of ${\rm Aut}^{+}(\Lambda)$ on $C(V)$ extends as representation of ${\rm Aut}^{+}(\Lambda)$ on all path spaces of finite lengths.
\end{proof}
\begin{proposition}
    Let $\lambda,\eta\in\Lambda^{\ast}$ such that $d(\lambda)=d(\eta)$. Then 
    \begin{eqnarray}
    \label{inner product preserving same length}
     \langle U(\chi_{[\lambda]}),U(\chi_{[\eta]})\rangle=\langle \chi_{[\lambda]}, \chi_{[\eta]} \rangle.   
    \end{eqnarray}
\end{proposition}
\begin{proof}
\indent Fix some $k>0$ and two elements $\lambda,\eta\in\Lambda^{\ast}$ such that $d(\lambda)=d(\eta)=k$. Write $\lambda$ and $\eta$ as $\lambda_{1}\lambda_{2}\cdots\lambda_{k}$ and $\eta_{1}\eta_{2}\cdots\eta_{k}$ for appropriately composable edges $\lambda_{i}$'s and $\eta_{i}$'s respectively. Then 
\begin{equation*}
    \langle U(\chi_{[\lambda]}), U(\chi_{[\eta]})\rangle = \biggl<
    \begin{aligned}[t]
     &\sum\limits_{\zeta;d(\zeta)=k;\zeta=\zeta_{1}\cdots\zeta_{k}}\chi_{[\zeta]}\otimes q_{r(\zeta_{1})r(\lambda_{1})}q_{s(\zeta_{1})s(\lambda_{1})} \cdots q_{r(\zeta_{n})r(\lambda_{k})}q_{s(\zeta_{k})s(\lambda_{k})},\\
     &\sum\limits_{\xi;d(\xi)=k;\xi=\xi_{1}\cdots\xi_{k}}\chi_{[\xi]}\otimes q_{r(\xi_{1})r(\eta_{1})}q_{s(\xi_{1})s(\eta_{1})} \cdots q_{r(\xi_{k})r(\eta_{k})}q_{s(\xi_{k})s(\eta_{k})} \biggl>
    \end{aligned}
\end{equation*}
\begin{equation*}
    = \sum\limits_{\zeta,\xi} \langle \chi_{[\zeta]}, \chi_{[\xi]} \rangle\otimes 
    \begin{aligned}[t]
        &q_{s(\zeta_{k})s(\lambda_{k})}q_{r(\zeta_{k})r(\lambda_{k})} \cdots q_{s(\zeta_{1})s(\lambda_{1})}q_{r(\zeta_{1})r(\lambda_{1})}q_{r(\xi_{1})r(\eta_{1})}q_{s(\xi_{1})s(\eta_{1})} \\
        &q_{r(\xi_{2})r(\eta_{2})}q_{s(\xi_{2})s(\eta_{2})}\cdots q_{r(\xi_{k})r(\eta_{k})}q_{s(\xi_{k})s(\eta_{k})}
    \end{aligned}
\end{equation*}
As $d(\zeta)=d(\xi)$, $\langle \chi_{[\zeta]}, \chi_{[\xi]} \rangle=0$ for $\zeta\neq \xi$, the last expression becomes
\begin{equation*}
  =\sum\limits_{\zeta}
  \begin{aligned}[t]
      &x_{s(\zeta)}^{\Lambda} q_{s(\zeta_{k})s(\lambda_{k})}q_{r(\zeta_{k})r(\lambda_{k})} \cdots q_{s(\zeta_{1})s(\lambda_{1})}q_{r(\zeta_{1})r(\lambda_{1})}q_{r(\zeta_{1})r(\eta_{1})}q_{s(\zeta_{1})s(\eta_{1})} \cdots \\
      &\cdots q_{r(\zeta_{k})r(\eta_{k})}q_{s(\zeta_{k})s(\eta_{k})} 
  \end{aligned}
\end{equation*}
Therefore, to prove \ref{inner product preserving same length}, one needs to show that
\begin{eqnarray*}
\delta_{\lambda,\eta}x^{\Lambda}_{s(\lambda)}&=\delta_{\lambda,\eta}\sum\limits_{\zeta}x_{s(\zeta)}^{\Lambda} q_{s(\zeta_{k})s(\lambda_{k})}q_{r(\zeta_{k})r(\lambda_{k})} \cdots q_{s(\zeta_{1})s(\lambda_{1})}q_{r(\zeta_{1})r(\lambda_{1})}q_{r(\zeta_{1})r(\eta_{1})}q_{s(\zeta_{1})s(\eta_{1})} 
\\&\;\;\;\;q_{r(\zeta_{2})r(\eta_{2})}q_{s(\zeta_{2})s(\eta_{2})} \cdots q_{r(\zeta_{k})r(\eta_{k})}q_{s(\zeta_{k})s(\eta_{k})}. 
\end{eqnarray*}
Now the above equation follows from the invariance of the unique KMS state under the action $\alpha$ from Proposition \ref{actionofqaut}. The reader can compare with the relevant expression from the proof of Proposition 2.31 in \cite{joardar2}. 
\end{proof}
\begin{lemma}
\label{inner product preserving on K}
    The $\mathbb{C}$ linear map $U: \mathcal{K}\raro \mathcal{K}\otimes_{alg} C(\text{Aut}^{+}(\Lambda))_{0}$ in theorem \ref{U-map} is inner product preserving on $\clk$. 
\end{lemma}
\begin{proof}
    Consider paths $\lambda,\eta$ of length $k,l$ respectively. Then $\chi_{[\lambda]},\chi_{[\eta]}\in \mathcal{R}_{{\rm max}\{k,l\}}$. Therefore, we can write $\chi_{[\lambda]},\chi_{[\eta]}$ as linear combinations of basis elements in $\mathcal{R}_{{\rm max}\{k,l\}}$. But a basis for $\clr_{{\rm max}\{k,l\}}$ is given by $\clc_{{\rm max}\{k,l\}}=\{\chi_{[\zeta]}:d(\zeta)={\rm max}\{k,l\}\}$. As $U$ is inner product preserving on $\clc_{{\rm max}\{k,l\}}$ by Proposition \ref{inner product preserving same length}, $\langle U(\chi_{[\lambda]}),U(\chi_{[\eta]})\rangle=\langle\chi_{[\lambda]},\chi_{[\eta]}\rangle1$. Therefore $U$ preserves the inner product on $\clk$.  
\end{proof}
As $\mathcal{K}$ is dense in $\mathcal{H}$ and $U:\mathcal{K}\raro \mathcal{K}\otimes_{alg} C(\text{Aut}^{+}(\Lambda))_{0}$ is inner product preserving, $U$ extends to an inner product preserving $\mathbb{C}$-linear map from $\clh$ to the Hilbert module $\clh\overline{\otimes}C(\mathbb{G}) $. We continue to denote the extension by $U$. Now we move forward to prove the density condition (III) of Definition \ref{Unitary representation definition}.
\begin{proposition}
\label{Udensity}
    For the inner product preserving map $U:\mathcal{H}\raro \mathcal{H} \overline{\otimes} C(\text{Aut}^{+}(\Lambda))$,
${\rm Sp}\{U(\xi).q\;|\; \xi\in \mathcal{H}, q\in C(\text{Aut}^{+}(\Lambda))\}\; \text{is dense in}\; \mathcal{H}\overline{\otimes} C(\text{Aut}^{+}(\Lambda))$.
\end{proposition}
\begin{proof}
Note that as the linear span of elements $\{\chi_{[\lambda]}:\lambda\in\Lambda^{\ast}\}$ is dense in $\clh$, it is enough to approximate the elements $\chi_{[\lambda]}\otimes 1$ for all $\lambda\in\Lambda^{\ast}$. When length of $\lambda$ is $0$ or $1$, this can be shown along the lines of \cite{weber}.
   % For a vertex $l\in V$ consider, 
%\begin{equation*}
 %  \sum\limits_{i\in V}U(\chi_{[i]})(1\otimes q_{li})=\sum\limits_{i\in V}\sum\limits_{j\in V}\chi_{[j]}\otimes q_{ji}q_{li}=\sum_{i}\chi_{[l]}\otimes q_{li}=\chi_{[l]}\otimes 1
%\end{equation*}
 Let us illustrate the proof for a path $\lambda$ of length two. For a path of any finite length the argument can be adapted. To that end let $\lambda=\gamma\beta$ be a two length path for edges $\gamma,\beta$ with $r(\beta)=s(\gamma)$. Then 
\begin{equation*}
    {\rm Sp}\{U(\xi).q\;\}\ni  \sum\limits_{i \in V}\sum\limits_{\substack{\zeta,\xi,r(\xi)=s(\zeta) \\
    r(\zeta\xi)=i\\ d(\xi)=d(\zeta)=1}} U(\chi_{[\zeta\xi]})( q_{s(\beta)s(\xi)}q_{r(\beta)r(\xi)}q_{s(\gamma)s(\zeta)}q_{r(\gamma)r(\zeta)})
\end{equation*}
Upon expanding, the last summation becomes
\begin{equation*}
    \sum\limits_{i\in V}\sum\limits_{\substack{\zeta,\xi\in \Lambda^{\ast} \\
r(\zeta\xi)=i}}\sum\limits_{\substack{\eta,\delta\in \Lambda^{\ast} \\r(\delta)=s(\eta)}} 
\begin{aligned}[t]
&\left(\chi_{[\eta\delta]}\otimes 
q_{r(\eta)r(\zeta)}q_{s(\eta)s(\zeta)}q_{r(\delta)r(\xi)}q_{s(\delta)s(\xi)}\right) \\
&( q_{s(\beta)s(\xi)}q_{r(\beta)r(\xi)}q_{s(\gamma)s(\zeta)}q_{r(\gamma)r(\zeta)})
\end{aligned}
\end{equation*}
\begin{equation*}
    =\sum\limits_{\substack{i\in V \\ \zeta,\xi \in \Lambda^{\ast}\\
    r(\zeta\xi)=i}}\sum\limits_{\eta,\delta} \chi_{[\eta\delta]} \otimes 
q_{r(\eta)r(\zeta)}q_{s(\eta)s(\zeta)}q_{r(\delta)r(\xi)}q_{s(\delta)s(\xi)}q_{s(\beta)s(\xi)}q_{r(\beta)r(\xi)}q_{s(\gamma)s(\zeta)}q_{r(\gamma)r(\zeta)}
\end{equation*}
For $s(\delta)\neq s(\beta)$, the corresponding coefficients in the above summation are zero and therefore the summation is equal to
\begin{equation*}
    \sum\limits_{i\in V}\sum\limits_{\substack{\zeta \\
r(\zeta)=i}}\sum\limits_{\substack{\eta,\delta,\xi \\ s(\delta)=s(\beta) \\ s(\eta)=r(\delta) \\ s(\zeta)=r(\xi)}}\chi_{[\eta\delta]} \otimes q_{r(\eta)r(\zeta)}q_{s(\eta)s(\zeta)}q_{r(\delta)r(\xi)}q_{s(\beta)s(\xi)}q_{r(\beta)r(\xi)}q_{s(\gamma)s(\zeta)}q_{r(\gamma)r(\zeta)}
\end{equation*}
Now in the above summation the term $q_{r(\delta)r(\xi)}q_{s(\beta)s(\xi)}q_{r(\beta)r(\xi)}$ can be replaced by $q_{r(\delta)k}q_{s(\beta)l}q_{r(\beta)k}$ for $(k,l)\in V\times V$. This is because if $(k,l)$ in not an edge, $q_{r(\delta)k}q_{s(\beta)l}=0$ by \ref{qaut_equ} as $s(\beta)=s(\delta)$. All the edges already appear in the summation. Therefore using the fact that $\sum_{l}q_{s(\beta)l}=1$, the summation reduces to 
\begin{equation*}
    \sum\limits_{i\in V}\sum\limits_{\substack{\zeta \\
r(\zeta)=i}}\sum\limits_{\substack{\eta,\delta,\xi \\ s(\delta)=s(\beta) \\ s(\eta)=r(\delta) \\ s(\zeta)=r(\xi)}}\chi_{[\eta\delta]} \otimes q_{r(\eta)r(\zeta)}q_{s(\eta)s(\zeta)}\Big(\sum\limits_{k\in V}q_{r(\delta)k}q_{r(\beta)k}\Big)q_{s(\gamma)s(\zeta)}q_{r(\gamma)r(\zeta)}.
\end{equation*}
%begin{equation*}
 %   \sum\limits_{i\in V}\sum\limits_{\substack{\zeta \\
%r(\zeta)=i}}\sum\limits_{\substack{\eta,\delta,\xi \\ s(\delta)=s(\beta) \\ s(\eta)=r(\delta) \\ s(\zeta)=r(\xi)}}\chi_{[\eta\delta]} \otimes q_{r(\eta)r(\zeta)}q_{s(\eta)s(\zeta)}\sum_{k}q_{r(\delta)k}q_{r(\beta)k}q_{s(\gamma)s(\zeta)}q_{r(\gamma)r(\zeta)}
%\end{equation*}
Again for $r(\delta)\neq r(\beta)$, $q_{r(\delta)k}q_{r(\beta)k}=0$ for all $k\in V$. Therefore using the fact that $\sum_{k}q_{r(\beta)k}=1$ and the graph has no multiple edges, the last summation further reduces to
%\begin{flalign*}
%\sum\limits_{\xi,s(\zeta)=r(\xi)}q_{r(\eta)r(\zeta)}q_{s(\eta)s(\zeta)}q_{r(\delta)r(\xi)}q_{s(\beta)s(\xi)}q_{r(\beta)r(\xi)}q_{s(\gamma)s(\zeta)}q_{r(\gamma)r(\zeta)} &&
%\end{flalign*}
%which can further be reduced to,
%\begin{flalign*}
%&\sum\limits_{k,l} q_{r(\eta)r(\zeta)}q_{s(\eta)s(\zeta)}q_{r(\delta)k}q_{s(\beta)l}q_{r(\beta)k}q_{s(\gamma)s(\zeta)}q_{r(\gamma)r(\zeta)} && \\
%&= \sum\limits_{k} q_{r(\eta)r(\zeta)}q_{s(\eta)s(\zeta)}q_{r(\delta)k}q_{r(\beta)k}q_{s(\gamma)s(\zeta)}q_{r(\gamma)r(\zeta)} &&
%\end{flalign*}
%As $s(\delta)=s(\beta)$, if $r(\delta)\neq r(\beta)$ then the above sum is zero. Therefore the last summation becomes
\begin{flalign*}
    \sum\limits_{i\in V}\sum\limits_{\zeta,r(\zeta)=i}\sum_{\eta}\chi_{[\eta\beta]}\otimes q_{r(\eta)r(\zeta)}q_{s(\eta)s(\zeta)}q_{s(\gamma)s(\zeta)}q_{r(\gamma)r(\zeta)}. 
\end{flalign*}
Now applying the same argument for a single edge proved in \cite{weber}, we have 
\[
 \sum\limits_{i \in V}\sum\limits_{\substack{\zeta,\xi,r(\xi)=s(\zeta) \\
    r(\zeta\xi)=i}} U(\chi_{[\zeta\xi]})(1\otimes q_{s(\beta)s(\xi)}q_{r(\beta)r(\xi)}q_{s(\gamma)s(\zeta)}q_{r(\gamma)r(\zeta)})=\chi_{[\gamma\beta]}\otimes 1, 
\]
i.e. $\chi_{[\lambda]}\otimes 1\in \text{Sp}\{U(\xi).q:\xi\in\clh,q\in C({\rm Aut}^{+}(\Lambda)\}$
\end{proof}
Hence the $\mathbb{C}$-linear map $U$ satisfies the conditions (I), (II) and (III) following the Definition \ref{Unitary representation definition} and therefore, thanks to the discussion following Definition \ref{Unitary representation definition}, $\widetilde{U}\in\cll\big(\clh\overline{\otimes}C({\rm Aut}^{+}(\Lambda))\big)=\clm\big(\clk(\clh)\otimes C({\rm Aut}^{+}(\Lambda))\big)$ is a unitary co-representation of ${\rm Aut}^{+}(\Lambda)$ on the Hilbert space $\clh$. 
%Now with this co-representation in hand, we can proceed to prove the main theorem.\\
%\indent {\it Proof of Theorem \ref{maintheorem}}: 
\begin{proposition}
    \label{commutation with D}
    The unitary co-representation $\widetilde{U}$ commutes with the Dirac operator $D$ i.e. $\widetilde{U}\circ(D\otimes 1)=(D\otimes 1)\circ\widetilde{U}$ on appropriate domains.
\end{proposition}
\begin{proof}
    Note that it is enough to show that $\widetilde{U}$ commutes with the projections on the eigen spaces of $D$ viz. $\Xi_{q,q-1}$ for all $q$ (see the discussion after Proposition \ref{R_kdim} for the notations). It is clear from the formula \ref{Uformula k lenghth path} that $U$ preserves the closed subspaces $\clr_{q}$ for all $q$. Therefore $\widetilde{U}$ commutes with the projections $\Xi_{q}$ for all $q$. Hence $\widetilde{U}$ being a unitary co-representation, commutes with $\Xi_{q-1}^{\perp}$ for all $q$. Then $\widetilde{U}$ commutes with $\Xi_{q}\wedge\Xi_{q-1}^{\perp}$ which is by definition $\Xi_{q,q-1}$. The last claim follows from the fact that $\Xi_{q}\wedge\Xi_{q-1}^{\perp}={\rm SOT-lim}_{n}(\Xi_{q}\Xi_{q-1}^{\perp}\Xi_{q})^n$.  
\end{proof}
\begin{proposition}
    \label{implementation}
    The action $\alpha$ of ${\rm Aut}^{+}(\Lambda)$ on $C^{\ast}(\Lambda)$ as given in Proposition \ref{actionofqaut} is implemented on the spectral data $(C^{\ast}(\Lambda),\clh,D)$ by the unitary co-representation $\widetilde{U}$.
\end{proposition}
\begin{proof}
We start by showing that for any $\lambda,\eta\in\Lambda^{\ast}$, 
\begin{equation}
\label{implementation_equation}
   (\pi\otimes [.]) \alpha(S_{\lambda}^{\ast})U(\chi_{[\eta]})=U(\pi(S_{\lambda}^{\ast})\chi_{[\eta]})
\end{equation}
 \textbf{Case1:} Consider the case $d(\lambda)\geq d(\eta)$. Then there are two subcases: Firstly let $\lambda=\eta\beta$ for some path $\beta$ of finite length. Let's write $\eta=\eta_{1}\eta_{2}\ldots\eta_{m}$ and $\lambda=\lambda_{1}\lambda_{2}\cdots\lambda_{n}$. Then by the unique factorization of paths, $\lambda$ can be written as $\eta_{1}\eta_{2}\ldots\eta_{m}\lambda_{m+1}\ldots\lambda_{n}$ where $\eta_{i}$ and $\lambda_{i}$'s are appropriate composable edges. Note that $s(\lambda)=s(\lambda_{n})$. Then by \ref{rep-2},  
\begin{displaymath}
U(\pi(S_{\lambda}^{\ast})\chi_{[\eta]})=\rho(\Lambda)^{-d(\lambda)/2}U(\chi_{[s(\lambda)]})=\rho(\Lambda)^{-d(\lambda)/2}\sum\limits_{i\in V}\chi_{[i]}\otimes q_{is(\lambda)}\end{displaymath}
 By the Formula \ref{Uformula k lenghth path},
\begin{displaymath}
U(\chi_{[\eta]})=\left(\sum\limits_{\zeta:\zeta=\zeta_{1}\zeta_{2}\ldots\zeta_{m}}\chi_{[\zeta]}\otimes q_{r(\zeta_{1})r(\eta_{1})}q_{s(\zeta_{1})s(\eta_{1})}\cdots q_{r(\zeta_{m})r(\eta_{m})}q_{s(\zeta_{m})s(\eta_{m})}\right)
\end{displaymath}
Therefore, 
\begin{equation*}
    \begin{aligned}
        &(\pi\otimes [.]) \alpha(S_{\lambda}^{\ast})U(\chi_{[\eta]})=  (\pi\otimes [.])(\alpha(S_{\lambda_{n}})^{\ast}\cdots \alpha(S_{\lambda_{1}})^{\ast}) U(\chi_{[\eta]})\\
&=(\pi\otimes [.])\left(\left(\sum\limits_{\xi_{n}}S_{\xi_{n}}^{*}\otimes q_{s(\xi_{n})s(\lambda_{n})}q_{r(\xi_{n})r(\lambda_{n})}\right)\cdots \left(\sum\limits_{\xi_{1}}S_{\xi_{1}}^{\ast}\otimes q_{s(\xi_{1})s(\lambda_{1})}q_{r(\xi_{1})r(\lambda_{1})}\right)\right)\\
&\;\;\;\;\left(\sum\limits_{\zeta}\chi_{[\zeta]}\otimes q_{r(\zeta_{1})r(\eta_{1})}q_{s(\zeta_{1})s(\eta_{1})}\cdots q_{r(\zeta_{m})r(\eta_{m})}q_{s(\zeta_{m})s(\eta_{m})}\right)\\
&=(\pi\otimes [.])\left(\sum\limits_{\substack{\xi=\xi_{1}\cdots\xi_{n}}}S_{\xi}^{*}\otimes q_{s(\xi_{n})s(\lambda_{n})}q_{r(\xi_{n})r(\lambda_{n})}\cdots  q_{s(\xi_{1})s(\lambda_{1})}q_{r(\xi_{1})r(\lambda_{1})} \right)\\
&\;\;\;\;\left(\sum\limits_{\zeta}\chi_{[\zeta]}\otimes q_{r(\zeta_{1})r(\eta_{1})}q_{s(\zeta_{1})s(\eta_{1})}\cdots q_{r(\zeta_{m})r(\eta_{m})}q_{s(\zeta_{m})s(\eta_{m})}\right)
    \end{aligned}
\end{equation*}
\begin{equation*}\hfill
    \begin{aligned}
&=\left(\sum\limits_{\substack{\xi=\xi_{1}\cdots\xi_{n}}}\pi(S_{\xi})^{\ast}\otimes [q_{s(\xi_{n})s(\lambda_{n})}q_{r(\xi_{n})r(\lambda_{n})}\cdots  q_{s(\xi_{1})s(\lambda_{1})}q_{r(\xi_{1})r(\lambda_{1})}] \right)\\
&\;\;\;\;\left(\sum\limits_{\zeta}\chi_{[\zeta]}\otimes q_{r(\zeta_{1})r(\eta_{1})}q_{s(\zeta_{1})s(\eta_{1})}\cdots q_{r(\zeta_{m})r(\eta_{m})}q_{s(\zeta_{m})s(\eta_{m})}\right)
\\
\end{aligned}
\end{equation*} 
\begin{equation*}\hfill
    =\sum\limits_{\substack{\xi,\zeta\\ d(\xi)\geq d(\zeta)}}\pi(S_{\xi})^{*}\chi_{[\zeta]}\otimes 
\begin{aligned}[t]
&q_{s(\xi_{n})s(\lambda_{n})}q_{r(\xi_{n})r(\lambda_{n})}\cdots  q_{s(\xi_{1})s(\lambda_{1})}q_{r(\xi_{1})r(\lambda_{1})}q_{r(\zeta_{1})r(\eta_{1})}\\
&q_{s(\zeta_{1})s(\eta_{1})}\cdots q_{r(\zeta_{m})r(\eta_{m})}q_{s(\zeta_{m})s(\eta_{m})}
\end{aligned}
\end{equation*}
Note that $d(\xi)\geq d(\zeta)$ as the action $\alpha$ and the map $U$ are both degree preserving. Now for any path $\xi,\zeta$ with $d(\xi)\geq d(\zeta)$ such that $\xi=\zeta\gamma$ for some path $\gamma$, $\pi(S_{\xi}^{\ast})(\chi_{[\zeta]})=\chi_{[s(\xi)]}$. If there is no path $\gamma$ of length $(n-m)$ such that $\xi=\zeta\gamma$, then $\pi(S_{\xi}^{\ast})(\chi_{[\zeta]})=0$. For $\xi=\zeta\gamma$, using the unique factorisation of paths, we write $\xi=\zeta_{1}\zeta_{2}\cdots \zeta_{m}\xi_{m+1}\cdots \xi_{n}$, $\gamma=\xi_{m+1}\xi_{m+2}\cdots \xi_{n}$. Note that $s(\xi)=s(\xi_{n})$. Therefore, using the expression $\lambda=\eta_{1}\cdots \eta_{m}\lambda_{m+1}\cdots\lambda_{n}$, the last summation becomes
\begin{equation*}
    \rho(\Lambda)^{-d(\xi)/2}\sum\limits_{\substack{\xi,\zeta\\ d(\xi)\geq d(\zeta) \\ \xi=\zeta\gamma}}\chi_{[s(\xi)]}\otimes 
\begin{aligned}[t]
&q_{s(\xi_{n})s(\lambda_{n})}q_{r(\xi_{n})r(\lambda_{n})}
\cdots q_{s(\xi_{m+1})s(\lambda_{m+1})}q_{r(\xi_{m+1})r(\lambda_{m+1})}\\
&q_{s(\zeta_{m})s(\eta_{m})}q_{r(\zeta_{m})r(\eta_{m})}
\cdots
q_{s(\zeta_{1})s(\eta_{1})}q_{r(\zeta_{1})r(\eta_{1})}\\
&q_{r(\zeta_{1})r(\eta_{1})}q_{s(\zeta_{1})s(\eta_{1})}\cdots\cdots q_{r(\zeta_{m})r(\eta_{m})}q_{s(\zeta_{m})s(\eta_{m}).}
\end{aligned}
\end{equation*}
As $d(\xi)=d(\lambda)$, ignoring the common scalar factor $\rho(\Lambda)^{-d/2}$, we shall show that for a fixed $i\in V$, the coefficient of $\chi_{[i]}$ in the above summation is $q_{is(\lambda)}$ proving \ref{implementation_equation} in this subcase. The coefficient is given by
\begin{equation*}
    \sum\limits_{\substack{\xi,\zeta\\ n=d(\xi)\geq d(\zeta)=m \\ \xi=\zeta\gamma, s(\xi)=i}}
\begin{aligned}[t]
&q_{s(\xi_{n})s(\lambda_{n})}q_{r(\xi_{n})r(\lambda_{n})}
\cdots q_{s(\xi_{m+1})s(\lambda_{m+1})}q_{r(\xi_{m+1})r(\lambda_{m+1})}q_{s(\zeta_{m})s(\eta_{m})}q_{r(\zeta_{m})r(\eta_{m})}
\\
&\cdots q_{s(\zeta_{1})s(\eta_{1})}q_{r(\zeta_{1})r(\eta_{1})}q_{r(\zeta_{1})r(\eta_{1})}q_{s(\zeta_{1})s(\eta_{1})}\cdots q_{r(\zeta_{m})r(\eta_{m})}q_{s(\zeta_{m})s(\eta_{m})}
\end{aligned}
\end{equation*}
Now consider the term $q_{s(\zeta_{1})s(\eta_{1})}q_{r(\zeta_{1})r(\eta_{1})}q_{s(\zeta_{1})s(\eta_{1})}$. As $\zeta_{1}$ varies over all edges such that $\zeta_{1}\zeta_{2}\ldots\zeta_{m}$ is an $m$-length path, as in the proof of \ref{Udensity}, we can replace this term in the summation with $\sum_{k,l\in V}q_{ks(\eta_{1})}q_{lr(\eta_{1})}q_{ks(\eta_{1})}$ which is $1$ as $\sum_{k}q_{ki}=1$ for any $i$. Therefore, the summation reduces to
\begin{equation*}
    \sum\limits_{\substack{\xi,\zeta\\ d(\xi)\geq d(\zeta) \\
s(\xi)=i}}
\begin{aligned}[t]
&\;q_{s(\xi_{n})s(\lambda_{n})}q_{r(\xi_{n})r(\lambda_{n})}\cdots
q_{s(\xi_{m+1})s(\lambda_{m+1})}q_{r(\xi_{m+1})r(\lambda_{m+1})}
\\
&\;q_{s(\zeta_{m})s(\eta_{m})}q_{r(\zeta_{m})r(\eta_{m})}\cdots q_{s(\zeta_{2})s(\eta_{2})}q_{r(\zeta_{2})r(\eta_{2})}q_{s(\zeta_{2})s(\eta_{2})}\\
&\;\cdots q_{r(\zeta_{m})r(\eta_{m})}q_{s(\zeta_{m})s(\eta_{m})}
\end{aligned}
\end{equation*}
By using the same trick repeatedly, the above summation can further be reduced to
\begin{equation*}
    \sum\limits_{\substack{\xi,
s(\xi)=i}}q_{s(\xi_{n})s(\lambda_{n})}q_{r(\xi_{n})r(\lambda_{n})}\cdots q_{s(\xi_{m+1})s(\lambda_{m+1})}q_{r(\xi_{m+1})r(\lambda_{m+1})}
\end{equation*}
Now note that in the above expression, $(r(\xi_{n}),s(\xi_{n-1}),\cdots, s(\xi_{m+1}),r(\xi_{m+1}))$ can be replaced by $(l_{n},k_{n-1},l_{n-1},\cdots,k_{m+1},l_{m+1})$ where $l_{i}'s$ and $k_{i}'s$ vary over the vertex set. This is because if $((k_{m+1},l_{m+1}),\cdots,(k_{n-1},l_{n-1}),(i,l_{n}))$ forms a path of degree $n-m$ starting at $i$ then the corresponding coefficients already appear in the summation. Otherwise, as $\lambda_{m+1}\cdots\lambda_{n}$ is an $(n-m)$-length path, \begin{displaymath}q_{is(\lambda_{n})}q_{l_{n}r(\lambda_{n})}\cdots q_{l_{m+1}r(\lambda_{m+1})}=0.\end{displaymath} Therefore, the coefficient of $\chi_{[i]}$ in the summation becomes
\begin{equation*}
    \sum\limits_{\substack{l_{n},k_{n-1},\ldots,k_{m+1},l_{m+1}\\s(\xi)=s(\xi_{n})=i}}q_{s(\xi_{n})s(\lambda_{n})}q_{l_{n}r(\lambda_{n})}\cdots q_{k_{m+1}s(\lambda_{m+1})}q_{l_{m+1}r(\lambda_{m+1})}=q_{is(\lambda)}.
\end{equation*}
Now consider the sub-case where there is no path $\beta$ such that $\lambda=\eta\beta$ which means there exists $k\in\{1,2,\cdots,m\}$ such that $\lambda_{k}\neq \eta_{k}$  and hence $\pi(S_{\lambda}^{\ast})\chi_{[\eta]}=0$. For simplicity let's assume that $\lambda_{1}\neq \eta_{1}$. The case $k>1$ can be proved along the same line of argument. Using the same computation as above the term $(\pi\otimes [.]) \alpha(S_{\lambda}^{\ast})U(\chi_{[\eta]})$ in the LHS is expanded as,
\begin{equation*}
    \rho(\Lambda)^{-d/2}\sum\limits_{\substack{\xi,\zeta\\ d(\xi)\geq d(\zeta) \\ \xi=\zeta\gamma}}\chi_{[s(\xi)]}\otimes 
\begin{aligned}[t]
&q_{s(\xi_{n})s(\lambda_{n})}q_{r(\xi_{n})r(\lambda_{n})}
\cdots q_{s(\xi_{m+1})s(\lambda_{m+1})}q_{r(\xi_{m+1})r(\lambda_{m+1})}\\
&q_{s(\zeta_{m})s(\lambda_{m})}q_{r(\zeta_{m})r(\lambda_{m})}
\cdots
q_{s(\zeta_{1})s(\lambda_{1})}q_{r(\zeta_{1})r(\lambda_{1})}q_{r(\zeta_{1})r(\eta_{1})}\\
&q_{s(\zeta_{1})s(\eta_{1})}\cdots\cdots q_{r(\zeta_{m})r(\eta_{m})}q_{s(\zeta_{m})s(\eta_{m})}.
\end{aligned}
\end{equation*}
As 
$\lambda_{1}\neq \eta_{1}$ that means either $r(\lambda_{1})\neq r(\eta_{1})$ or $s(\lambda_{1})\neq s(\eta_{1})$. Now for the case $r(\lambda_{1})\neq r(\eta_{1})$ the term $q_{r(\zeta_{1})r(\lambda_{1})}q_{r(\zeta_{1})r(\eta_{1})}=0$ otherwise lets assume that $r(\lambda_{1})= r(\eta_{1})$ and $s(\lambda_{1})\neq s(\eta_{1})$, 
now using the same argument as before, we can reduce the term $q_{s(\zeta_{1})s(\lambda_{1})}q_{r(\zeta_{1})r(\eta_{1})}q_{s(\zeta_{1})s(\eta_{1})}$ to $\sum\limits_{k,l}q_{ks(\lambda_{1})}q_{lr(\eta_{1})}q_{ks(\eta_{1})}=\sum\limits_{k}q_{ks(\lambda_{1})}q_{ks(\eta_{1})}=0$
and we are done.\\
\textbf{Case 2:} $d(\lambda)<d(\eta)$: Again as in case 1, it can be divided into two subcases: there is some $\beta$ such that $\lambda\beta=\eta$ and there is no $\beta$ such that $\lambda\beta=\eta$. In the first subcase, $\pi(S_{\lambda}^{\ast})(\chi_{[\eta]})=\chi_{[\beta]}$. In the second subcase $\pi(S_{\lambda}^{\ast})(\chi_{[\eta]})=0$. In both the cases, using the same computational technique, it can be shown that 
\begin{displaymath}
    U(\pi(S_{\lambda}^{\ast})(\chi_{[\eta]}))=(\pi\otimes[.])\alpha(S_{\lambda}^{\ast})U(\chi_{[\eta]}).
\end{displaymath}

Therefore, for any path $\eta\in\Lambda^{\ast}$, we have $(\pi\otimes[.])\alpha(S_{\lambda}^{\ast})U(\chi_{[\eta]})=U(\pi(S_{\lambda^{\ast}}))(\chi_{[\eta]})$ and hence for any $q\in C({\rm Aut}^{+}(\Lambda))$,
\begin{displaymath}
  (\pi\otimes[.])\alpha(S_{\lambda}^{\ast})\widetilde{U}(\chi_{[\eta]}\otimes q)=\widetilde{U}(\pi(S_{\lambda}^{\ast})\otimes [1])(\chi_{\eta}\otimes q). 
\end{displaymath} 
By density of ${\rm Sp}\{\chi_{[\eta]}\otimes q:\eta\in\Lambda^{\ast}, q\in C({\rm Aut}^{+}(\Lambda))\}$ in the Hilbert space $\clh\otimes L^{2}({\rm Aut}^{+}(\Lambda))$ and unitarity of $\widetilde{U}$, we have for any $\lambda\in\Lambda^{\ast}$,
\begin{eqnarray}
\label{implementation star}
    (\pi\otimes[.])\alpha(S_{\lambda}^{\ast})=\widetilde{U}(\pi(S_{\lambda}^{\ast})\otimes[1])\widetilde{U}^{\ast}\in \clb(\clh\otimes L^{2}({\rm Aut}^{+}(\Lambda))).
\end{eqnarray}
Using the same computational technique it can be shown that 
\begin{eqnarray}
\label{implementation non star}
    (\pi\otimes[.])\alpha(S_{\lambda})=\widetilde{U}(\pi(S_{\lambda})\otimes[1])\widetilde{U}^{\ast}\in \clb(\clh\otimes L^{2}({\rm Aut}^{+}(\Lambda)))
\end{eqnarray}
for any $\lambda\in\Lambda^{\ast}$. Alternatively, one can see as before that Equation \ref{implementation non star} is equivalent to the following for $\lambda,\eta\in\Lambda^{\ast}$:
\begin{equation}
\label{non star}
    (\pi\otimes[.])\alpha(S_{\lambda})U(\chi_{[\eta]})=U(\pi(S_{\lambda}))(\chi_{[\eta]}).
\end{equation}
Again recalling the representation $\pi$ of $C^{\ast}(\Lambda)$, we see that RHS of the above is equal to $\rho(\Lambda)^{d(\lambda/2)}U(\chi_{[\lambda\eta]})$ whereas the LHS is equal to 
\begin{equation}
\label{temp}    
\begin{aligned}
    &\rho(\Lambda)^{d(\mu)/2}\sum_{}\chi_{[\mu\xi]}\otimes q_{r(\mu_{1})r(\lambda_{1})}q_{s(\mu_{1})s(\lambda_{1})}\cdots q_{r(\mu_{n})r(\lambda_{n})}q_{s(\mu_{n})s(\lambda_{n})}q_{r(\xi_{1})r(\eta_{1})}\\
    &q_{s(\xi_{1})s(\eta_{1})}\cdots 
    q_{r(\xi_{m})r(\eta_{m})}q_{s(\xi_{m})s(\eta_{m})}
\end{aligned}
\end{equation} 
where $\mu=\mu_{1}\ldots\mu_{n}$ and $\xi=\xi_{1}\ldots\xi_{m}$;
\begin{displaymath}  
\alpha(S_{\lambda})=\sum S_{\mu}\otimes q_{r(\mu_{1})r(\lambda_{1})}q_{s(\mu_{1})s(\lambda_{1})}\cdots q_{r(\mu_{n})r(\lambda_{n})}q_{s(\mu_{n})s(\lambda_{n})}
\end{displaymath} 
and 
\begin{displaymath}
U(\chi_{[\eta]})=\sum\limits_{\xi}\chi_{[\xi]}\otimes q_{r(\xi_{1})r(\eta_{1})}q_{s(\xi_{1})s(\eta_{1})}\cdots \\
    q_{r(\xi_{m})r(\eta_{m})}q_{s(\xi_{m})s(\eta_{m})}.
\end{displaymath} 
A moment's reflection reveals that the coefficients of $\chi_{[\mu\xi]}$ in \ref{temp} are the coefficients of $S_{\mu}S_{\xi}$ in the expression of $\alpha(S_{\lambda})\alpha(S_{\eta})$. Then the homomorphic property of $\alpha$ coupled with the fact that $d(\mu)=d(\lambda)$ for all $\mu$ proves Equation \ref{non star}.  
Since both $(\pi\otimes[.])\alpha$ and ${\rm ad}_{\widetilde{U}}$ are $C^{\ast}$-homomorphisms, we have the following from \ref{implementation star} and \ref{implementation non star}:
\begin{displaymath}
 (\pi\otimes[.])\alpha(S_{\lambda}S_{\mu}^{\ast})=\widetilde{U}(\pi(S_{\lambda}S_{\mu}^{\ast})\otimes[1])\widetilde{U}^{\ast}\in \clb(\clh\otimes L^{2}({\rm Aut}^{+}(\Lambda))),   
\end{displaymath}
for $\lambda,\mu\in\Lambda^{\ast}$. As ${\rm Sp}\{S_{\lambda}S_{\mu}^{\ast}:\lambda,\mu\in\Lambda^{\ast}\}$ is dense in $C^{\ast}(\Lambda)$, using the continuity of the maps involved, we get
\begin{displaymath}
 (\pi\otimes[.])\alpha(a)=\widetilde{U}(\pi(a)\otimes[1])\widetilde{U}^{\ast}\in \clb(\clh\otimes L^{2}({\rm Aut}^{+}(\Lambda))),   
\end{displaymath}
for all $a\in C^{\ast}(\Lambda)$, proving the proposition.
\end{proof} 
 {\it Proof of Theorem \ref{maintheorem}}: Combining Proposition \ref{commutation with D} and Proposition \ref{implementation}, we see that the action $\alpha$ satisfies conditions (I) and (II) of Definition \ref{quantumisometricdefinition}. This completes the proof of the main theorem.\qed
\section{A non-isometric action on the Cuntz algebra}
We begin this section with an observation from the foregoing section. Note that in the proof of unitarity of $\widetilde{U}$, the invariance of KMS-state under the given action $\alpha$ was crucially used (see Lemma \ref{inner product preserving on K}). So as mentioned in the introduction, it is natural to ask the question whether any object in the category of CQG's acting linearly on the graph $C^{\ast}$-algebra of a strongly connected graph in the sense of \cite{joardar1} such that the KMS state is invariant act isometrically or not. In this section we give a counter example. To that end let us recall that the quantum unitary group $U_{n}^{+}$ acts linearly on the Cuntz algebra $\clo_{n}$ and it is the universal object in the category of CQG's acting linearly on $\clo_{n}$ such that the unique KMS-sate is invariant. Let us recall the action $\alpha:\clo_{n}\raro\clo_{n}\otimes C(U_{n}^{+})$ given by (see \cite{joardar1})
\begin{equation}
    \label{Cuntz action}
    \alpha(S_{e_{i}})=\sum_{j=1}^{n}S_{e_{j}}\otimes q_{ji} 
\end{equation}
where $\{e_{1},\ldots,e_{n}\}$ are the loops of the graph of the Cuntz algebra ; $S_{e_{i}}$'s are the generating partial isometries; and $\{q_{ij}\}_{i,j=1,\ldots,n}$ are the generators of $C(U_{n}^{+})$. In the following we are going to retain the standard notations used in the previous section.
\begin{theorem}
    The action $\alpha$ of $U_{n}^{+}$ on $\clo_{n}$ given by \ref{Cuntz action} is not isometric in the sense of Definition \ref{quantumisometricdefinition}.
\end{theorem}
\begin{proof}
    To prove the theorem we have to show that there is no unitary co-representation of $U_{n}^{+}$ on the Hilbert space of the infinite path space (say $\clh$ as before) commutting with the Dirac operator can implement the given action. To that end let us take a unitary co-representation $\widetilde{U}$ that commutes with the Dirac operator $D$. We shall work with the corresponding $\mathbb{C}$-linear map $U:\clh\raro\clh\overline{\otimes}C(U_{n}^{+})$. Then $U$ must preserve the eigen space $\clr_{0}$ of $D$ which is one-dimensional and spanned by the constant function $1$ which is $\chi_{[v]}$ where $v$ is the only vertex of the graph.  Then $U(1)=1\otimes q$ for some unitary $q\in C(U_{n}^{+})$. If possible, let $U$ implement $\alpha$. Then we have $(\pi\otimes [.])\alpha(S_{e_{i}})U(1)=U(\pi(S_{e_{i}}).1)=U(\chi_{[e_{i}]})$. By expanding $\alpha(S_{e_{i}})$, we get
\begin{eqnarray*}
    (\pi\otimes [.])\alpha(S_{e_{i}})U(1)
    &=&(\sum\limits_{j}\pi(S_{e_{j}})\otimes [q_{ji}])(1\otimes q)\\
    &=&\sum\limits_{j} \pi(S_{e_{j}}).1\otimes q_{ji}q \\
    &=&\sum\limits_{j} \chi_{[e_{j}]}\otimes q_{ji}q 
\end{eqnarray*}
Thus we have $U(\chi_{[e_{i}]})=\sum\limits_{j} \chi_{[e_{j}]}\otimes q_{ji}q$. But $1=\sum\limits_{j}\chi_{[e_{j}]}$ and therefore $U(1)=U(\sum\limits_{i}\chi_{[e_{i}]})=\sum\limits_{i,j}\chi_{[e_{j}]}\otimes q_{ji}q$. As $U(1)=1\otimes q$, we get 
\begin{equation}\label{contradict equation}\sum\limits_{j}\chi_{[e_{j}]}\otimes \left(\sum\limits_{i}q_{ji}q \right)=\sum\limits_{j}\chi_{[e_{j}]}\otimes q.\end{equation} Now choose a point $x\in\Lambda^{\infty}$ such that $x\in[e_{k}]$ for a unique $k$. Then evaluating both sides of the Equation \ref{contradict equation} at the point $x$, we get
$\sum_{i}q_{ki}q=q$ which, by unitarity of $q$, implies that $\sum_{i}q_{ki}=1$, which is a contradiction as the generators of $C(U_{n}^{+})$ do not satisfy these relations.
\end{proof}
\begin{remark}
(1) The quantum permutation group $S_{n}^{+}$ has a `linear' action $\alpha$ on $\clo_{n}$ given by 
  \begin{displaymath}
      \alpha(S_{e_{i}})=\sum_{j=1}^{n} S_{e_{j}}\otimes q_{ji},
  \end{displaymath}
  where $\{q_{ij}\}_{i,j=1,\ldots,n}$ are the generators of $C(S_{n}^{+})$. Then it can be verified that the unique KMS state of $\clo_{n}$ is invariant under the action $\alpha$. Using the same computational technique and same line of arguments it can be shown that the action $\alpha$ is also isometric in the sense of Definition \ref{quantumisometricdefinition}. The quantum group $S_{n}^{+}$ is a quantum symmetry object of the underlying graph of $\clo_{n}$ in the sense of Goswami et al. (\cite{asfaq}) where they consider quantum symmetry of finite graphs possibly with multiple edges. It seems reasonable to expect that the quantum symmetry of the underlying graph of a graph $C^{\ast}$-algebra is well behaved in the context of quantum isometric actions with respect to the class of spectral triple on graph $C^{\ast}$-algebras constructed by Farsi et al. in \cite{farsi}. On can explore further in this direction by considering graphs with possible multiple edges.\\
  \indent (2) Since the spectral triple is $\Theta$-summable for generic choices of eigenvalues of the Dirac operator, one can consider quantum isometry in terms of a `good Laplacian' in the sense of \cite{goswami2} and possibly compare with our consideration.
\end{remark}
{\bf Acknowledgement} The first author would like to acknowledge the support provided by SERB MATRICS grant (MTR/2022/000515) provided by the government of India. Both the authors would like to acknowledge the support provided by the FIST grant (SR/FST/MS-II/2019/51) provided by the government of India. The first author would like to thank Debashish Goswami for some helpful inputs.

\end{document}